\RequirePackage{amsthm}

\documentclass[sn-mathphys-num,pdflatex]{sn-jnl}%

\usepackage{graphicx}%
\usepackage{color}
\usepackage{multirow}%
\usepackage{amsmath,amssymb,amsfonts}%
\usepackage{amsthm}%
\usepackage{mathrsfs}%
\usepackage[title]{appendix}%
\usepackage{xcolor}%
\usepackage{textcomp}%
\usepackage{manyfoot}%
\usepackage{booktabs}%
\usepackage{algorithm}%
\usepackage{algorithmicx}%
\usepackage{algpseudocode}%
\usepackage{listings}%
\usepackage{placeins}
\usepackage{paralist}
\usepackage{transparent}
\usepackage{subfigure}%

\theoremstyle{thmstyleone}%
\newtheorem{theorem}{Theorem}%
\newtheorem{proposition}[theorem]{Proposition}%
\newtheorem{corollary}[theorem]{Corollary}%
\newtheorem{lemma}[theorem]{Lemma}%
\newtheorem{case}{Case}

\theoremstyle{thmstyletwo}%
\newtheorem{example}{Example}%
\newtheorem{remark}{Remark}%

\theoremstyle{thmstylethree}%

\newcommand{\R}{\mathbb{R}}   %
\newcommand{\PR}{\operatorname{P_a}}
\newcommand{\dom}{\operatorname{dom}}
\newcommand{\ran}{\operatorname{ran}}
\newcommand{\inte}{\operatorname{int}}
\newcommand{\ri}{\operatorname{ri}}
\newcommand{\co}{\operatorname{conv}}%
\newcommand{\cone}{\operatorname{cone}}%
\newcommand{\argmin}{\operatorname{argmin}}
\newcommand{\cartesian}{\ensuremath{\times}}
\newcommand{\eg}{\textit{e.g., }}
\newcommand{\ie}{\textit{i.e., }}

\usepackage{silence}
\begin{document}

\title[Conjugate of PLQ functions]{Computing the convex envelope of bivariate piecewise linear-quadratic functions in linear time}

 \author{\fnm{Tanmaya} \sur{Karmarkar}} %

 \author*[1]{\fnm{Yves} \sur{Lucet}}\email{yves.lucet@ubc.ca}
 \equalcont{These authors contributed equally to this work.}

 \affil*[1]{\orgdiv{Computer Science, CMPS, I. K. Barber Faculty of Science}, \orgname{UBC Okanagan}, \orgaddress{\street{3187 University Way}, \city{Kelowna}, \postcode{V1T 1T7}, \state{BC}, \country{Canada}}}

 \abstract{We compute the convex envelope of (nonconvex) bivariate piecewise linear-quadratic (PLQ) functions (bivariate quadratic functions defined on a polyhedral subdivision). Our algorithm is composed of the following steps: (1) compute the convex envelope of each quadratic piece obtaining piecewise rational functions (quadratic divided by linear function) defined over a polyhedral subdivision; (2) compute the (Legendre-Fenchel) conjugate of each resulting piece to obtain piecewise quadratic functions defined over a parabolic subdivision; (3) compute the maximum of all those functions to obtain the conjugate of the original PLQ function as a piecewise quadratic function defined on a parabolic subdivision; (4) compute the conjugate of each resulting piece; and finally (5) compute the maximum over all those functions to obtain the convex envelope (biconjugate) as rational functions (quadratic divided by linear function) defined over a polyhedral subdivision. Our contribution includes a practical algorithm running in linear time, and proving that the convex envelope is a piecewise rational function.}

\maketitle

\section{Introduction}\label{s:intro}%
There are several motivations for the present work. One goal is to obtain tight lower bounds for relaxation in global optimization. The convex envelope, which is equal to the biconjugate, provides the tightest relaxation when one wishes to solve a global optimization problem. The convex envelope also provides information on the set of minimizers. If we denote by $\argmin J$ the set of $x \in X$ minimizing $J$ on $X$ (possibly, the empty set), we easily see that
\[\overline{\co} (\argmin \ J) \subset \argmin \  \overline{\co} (J)),\]
with $\overline{\co}(J)$ denoting the closed convex envelope of $J$, with equality for example when  $\overline{\co} (J)$ is $0$-coercive~\cite[Theorem 6]{HIRIART-URRUTY-11}.
A large literature is available on global optimization and the search for tight relaxation; we refer to~\cite{LOCATELLI-13} (especially Section 4.2 on convex envelopes) for a detailed introduction and more details on relaxations relevant for our context.

Another goal is to understand the structure of the biconjugate of piecewise linear-quadratic (PLQ) functions (bivariate functions defined on a union of polyhedral set on each of which the restriction of the function is quadratic). Such understanding is of interest in itself. 

Computing the convex envelope of a function is a hard problem; even computing the convex envelope of a multilinear function over a unit hypercube is NP-Hard~\cite{CRAMA-89}. However, results for specific functions exist in the literature, in particular for quadratic bivariate polynomials~\cite{LOCATELLI-14,LOCATELLI-14a,LOCATELLI-16,LOCATELLI-18a,KHADEMNIA-24}, and for convex envelopes of bilinear functions over triangles, rectangles and special polytopes~\cite{AL-KHAYYAL-83, SHERALI-90, LINDEROTH-05, ANSTREICHER-10, ANSTREICHER-12}. It is an active subject of research~\cite{LOCATELLI-24}.

PLQ functions play a significant role in variational analysis~\cite[Section 10E, Section 11D, p. 440]{ROCKAFELLAR-98a} due to the availability of calculus rules~\cite[Example 11.28, Proposition 11.32, Corollary 11.33, Proposition 12.30]{ROCKAFELLAR-98a}, and due to their duality property~\cite[Theorem 11.42, Example 11.43, Theorem 11.42]{ROCKAFELLAR-98a}. Compared to piecewise linear functions, PLQ functions capture the curve of the original function more accurately%
. The set of convex PLQ functions is closed under common convex operators, in particular under the Legendre-Fenchel transform; see~\cite[Page 484]{ROCKAFELLAR-98a},~\cite[Proposition 3, p. 17]{GOEBEL-00} and \cite{GARDINER-13}. 

The early idea of the computation of convex transforms can be traced back to~\cite{MOREAU-65}. However, development of most of the algorithms in computational convex analysis began with the computation of the conjugate with the Fast Legendre Transform~\cite{BRENIER-89a}, which is studied in~\cite{CORRIAS-93a,LUCET-96a}. A linear-time algorithm is introduced in~\cite{LUCET-97b}. Those algorithms handle nonconvex functions but are restricted to grid domains. Extension to nongrid domains for convex functions only are proposed in~\cite{HAQUE-18,GARDINER-13} with a generalization to the partial conjugate in~\cite{GARDINER-13}. Computation of the conjugate of convex univariate PLQ functions has been well studied and linear time algorithms have been developed, both for the PLQ~\cite{LUCET-06} and the graph-matrix (GPH) data structures~\cite{GARDINER-11a}. All in all, we are not aware of algorithm to compute the conjugate of nonconvex PLQ functions over nongrid domains.

Implementation of conjugate computation algorithms can be found in the CCA library~\cite{LUCET-17,LUCET-21} that contains efficient algorithms to manipulate convex functions and to compute fundamental convex analysis transforms arising in the field of convex analysis. Algorithms to compute the conjugate numerically on grids have  been based on either parameterization~\cite{HIRIART-URRUTY-06}, manipulation of graphs (GPH model)~\cite{GARDINER-13}, or the computation of the Moreau envelope~\cite{LUCET-05c}. More complex operators such as the proximal average operator~\cite{BAUSCHKE-07a,BAUSCHKE-06a} can be built by using a combination of addition, scalar multiplication, and conjugacy operations. 

A numerical library to determine in linear time whether a PLQ function is convex~\cite{SINGH-21} is available in MATLAB~\cite{CCA2}. This library is the first to handle general piecewise quadratic functions that are not necessarily continuous and are defined on any polyhedral subdivision without approximating them with a grid.

Computational Convex Analysis has many applications in the fields of image processing, network communication, PDE, geographic information systems, computer aided design, molecular biology, medical imaging, computer graphics and robotics; see the survey~\cite{LUCET-10}.

We propose the first algorithm to compute the convex envelope of a general bivariate PLQ function not necessarily convex without approximating it with a grid. We assume the pieces are triangular; this assumption is not restrictive since we can always triangulate convex polyhedral sets in linear time~\cite[p. 49]{BERG-08}. Our algorithm is split in four steps. In Step 1 we compute the convex envelope of each piece of a PLQ function $f$ using formulae derived in Appendix~\ref{s:convexenvelope_formulae} using the method given in \cite{LOCATELLI-16}; in Step 2, we apply \cite{KUMAR-19} to compute the conjugate of each resulting piece; and in Step 3, we compute the maximum of all the resulting piecewise functions to obtain the conjugate $f^*$; refer to \cite{KARMARKAR-24}. Step 4 computes the biconjugate to obtain the convex envelope by first computing the conjugate of each piece and then computing the maximum of all the resulting piecewise functions.

Our contribution is threefold: 
 \begin{inparaenum}[(1)]
    \item developing formulae to compute in linear time the convex envelope and conjugate of each piece (this steps improves Locatelli's method~\cite{LOCATELLI-16} that has a worst case exponential time),
    \item proving the entire algorithm runs in linear time; and 
    \item releasing an open-source code that uses symbolic computation and rational arithmetic to avoid floating point errors.
\end{inparaenum}

We begin by giving some preliminaries in Section~\ref{s:prelim}, describing the structure of the conjugate of a PLQ function in Section~\ref{s:struct} and then the method to compute the domain and the conjugate expressions in Sections ~\ref{s:dom} and~\ref{s:expr}. We present examples in Section~\ref{s:examples}, and conclude with future work in Section~\ref{s:conclusion}.

\section{Preliminaries}\label{s:prelim}%
The \emph{domain} of a function $f:\R^n\to \R\cup\{+\infty\}$ is the set $\{x : f(x)<+\infty\}$, and the function is called \emph{proper} if it has nonempty domain. The \emph{indicator function} of $C \subset \R^n$ is denoted $I_C:\R \to \{0,\infty\}$ with $I_C(x)=0$  if $x \in C$ and $+\infty$ otherwise. 

A function is called a \emph{piecewise function} if it can be written $f(x)=f_i(x)$ for $x\in R_i$ with 
$\cup_i R_i=\R^n$. We call  the pair $(f_i,R_i)$ a piece. 
Recalling~\cite[Definition 10.20]{ROCKAFELLAR-98a}, a function $f:\R^n\to \R\cup\{+\infty\}$ is called \emph{Piecewise Linear-Quadratic (PLQ)} if there exists a finite number of polyhedral sets $R_i$ with $\dom f = \cup_i R_i$ and $f$ restricted to $R_i$, noted $f_i$, can be written $f_i(x)=1/2 x^T Q_i x + q_i^T x + \kappa_i$ with $\kappa_i\in \R$, $q_i\in\R^n$, and $Q_i\in\R^{n\times n}$ a symmetric matrix. We use column vectors with $q^T x$ denoting the dot product of row vector $q^T$ with column vector $x$.
PLQ functions enjoy several properties, \eg their domain is closed, they are lower semicontinuous, and continuous relative to their domain~\cite[Proposition 10.21]{ROCKAFELLAR-98a}.

An \emph{underestimator} is any function that lies below a given function. The \emph{closed convex envelope} is the largest lower semicontinuous convex underestimator of $f$.

We note 
\[f^*(y)=\sup_{x\in \R^n}(x^Ty - f(x))\] 
the \emph{(Legendre-Fenchel) conjugate} of $f$.  The \emph{biconjugate} is noted $f^{**}$, and we recall that when $f$ is proper, closed and convex, $f=f^{**}$, otherwise $f^{**}$ is the closed convex envelope of $f$~\cite[Theorem 11.1]{ROCKAFELLAR-98a}.    

To describe the domain of piecewise functions, we note a \emph{hyperplane} $H \subset \R^d$ a set of the form $H= \{x \in \R^d: \alpha^Tx - \beta = 0\}$, and a \emph{halfspace} $H \subset \R^d$ a set of the form $H= \{x \in \R^d: \alpha^T x - \beta \le 0\}$ or $H= \{x \in R^d: \alpha^Tx - \beta \ge 0\}$; where the vector $\alpha \in \R^d \backslash  \{0\}$ is called a normal vector to $H$.

A \emph{polytope} is a convex combination of finite numbers of vertices $\{v_1, \dots, v_k\}$. We write $
\co V = \{ \lambda_1 v_1 + \dots + \lambda_k v_k: \lambda_i \ge 0, \sum_{i=1}^{k}\lambda_i=1 \}$ or in matrix notation $\co V =  \{ V\lambda: \lambda_i \ge 0, \sum_{i=1}^{k}\lambda_i=1 \}$,
where $V$ is the matrix formed by column vectors. Note that by our definition, polytopes are always bounded.

A \emph{polyhedral cone} is defined as a conic combination of a finite numbers of directions $\{d_1,\dots,d_m\}$; we write $\cone C = \{\mu_1d_1+ \dots + \mu_md_m : \mu_i \ge 0\}$, and $\cone C = \{D\mu : \mu_i \ge 0\}$;
where $D$ is the matrix formed by column vectors $d_i$.

A \emph{polyhedral set (polyhedron)} $P \subset \R^d$ is defined as the intersection of a finite number of closed halfspaces and hyperplanes; it has the $H$-representation 
\[P=\{x\in \R^d: Ax\leq b\}\] 
with $A\in\R^{m\times d}$ and $b\in \R^m$. 
We can also describe a polyhedral set in Vertex-representation. According to the Minkowski-Weyl Theorem~\cite[Corollary 3.53]{ROCKAFELLAR-98a}, a polyhedral set can be written as a sum of a polytope and a polyhedral cone. Therefore, in $V$-representation a polyhedral set $P$ is described as
\[
	P = \co V + \cone C 
	  = \bigg\{ V\lambda + D\mu: \lambda \in \R^k, \lambda \ge 0, \mu \ge 0, \sum_{i=1}^{k}\lambda_i=1 \bigg\}.
\]
Since the intersection of convex sets is convex, polyhedral sets are convex.

A $d$-dimensional face of a convex set $C\subset \R^n$ is a $d$-dimensional convex subset $C'$ of $C$ such that every (closed) line segment in C with a relative interior point in $C'$ has both endpoints in $C'$~\cite[Page 162]{ROCKAFELLAR-70}. The empty set and $C$ itself are faces of $C$. The two-dimensional faces of $C$ are called faces.
The one-dimensional faces of a convex set C are called edges. For $u,v\in\R^d, u \neq v$  an edge can be written as
\[\bigg\{x\in\R^d :   x=\delta_1u+\delta_2v, \delta_1+\delta_2=1, (\delta_1,\delta_2)\in \Delta\subset\R^2\bigg\}.\]
An edge in $\R^2$ is a segment if $\Delta = \R_+ \cartesian \R_+ $ where $\R_+= \{\alpha\in \R:\alpha\ge 0\}$, a ray if $\Delta = \R_+ \cartesian \R$, and a line if $\Delta = \R^2$.
The zero-dimensional faces of a convex set $C$ are called vertices and are the extreme points of $C$.

A \emph{convex} set defined as the union of a finite number of polyhedral regions, namely
$R = \bigcup_{i=1}^n R_{i}$,
where $R \subseteq \R^2$, is said to be a polyhedral subdivision~\cite[Definition 1]{PATRINOS-11}, if $R_i$ is a polyhedral set and for any  \(j, k \in \{1, \dots, n\}, j \neq k, R_{j} \cap R_{k}\) is either empty, a vertex or an edge.
(Polyhedral subdivisions eliminate degenerate subdivisions for which the intersection of two polyhedral regions is a strict subset of an edge, or of a face.)

Assume that $f: \R^2 \rightarrow \R \cup \{+\infty\}$  is a PLQ function and that $\dom f = \cup_i P_i $ is a polyhedral subdivision. An entity is a $d$-dimensional face of $P_i$ for some index $i$. An entity is either a vertex, an edge or a face of the polyhedral subdivision of $\dom f$~\cite{GARDINER-11}.

We now define the geometric shapes required to describe the domain of the conjugate functions we obtain. 
A \emph{parabola} $\PR\subset\R^2$ is a subset of the plane that can be written as
\[\PR=\{(x,y)\in\R^2: ax^2+b x y+c y^2+ d x+e y +f = 0\},\]
where $a, b, c, d, e, f\in\R$ are not all zero and satisfy \(b^2 - 4ac = 0\) .
Note that our definition includes lines when $a=b=c=0$, and the empty set when $a=b=c=d=e=0$, $f\neq 0$; but excludes the entire plane since we impose that $(a,b,c,d,e,f)\neq 0$.

A \emph{parabolic region} \(\PR \subset  \R^2\) is formed by the intersection of a finite number of parabolic inequalities. It can be written as 
\[\PR = \{x \in \R^2: a_ix^2 + b_ixy + c_iy^2 + d_ix + e_iy + f_i \leq 0, \forall i=1, \dots , k\},\] 
where $(a_i,b_i,c_i,d_i,e_i,f_i)\neq 0$, and \( b_i^2- 4a_i c_i = 0\).
The following sets are special cases of parabolic regions: polyhedral sets, polyhedral cones, polytopes, edges, and vertices. The set $\{(x,y)\in\R^2: y\geq x^2\}$ is an example of a non-polyhedral parabolic region while the set $\{(x,y) : y\leq x^2, y\geq 1\}$ is a nonconvex non-connected parabolic region.

A convex set $R\subset \R^2$ is called a \emph{parabolic subdivision} if $R$ can be written as the finite union of parabolic regions $R_i$, \ie $R = \cup_{i=1}^n R_{i}$ and for any $j, k \in \{1, \dots , n\}$, $j\neq k$, the intersection $R_j \cap R_k$ is either empty or is contained in a parabola.
A polyhedral subdivision is a special case of a parabolic subdivision (since lines are a special case of parabolas).

A face of a parabolic region \(\PR \subset  \R^2\) is defined as the interior of a nonempty intersection of a finite number of parabolic inequalities. It can be written as 
\[\PR = \{x \in \R^2: a_ix^2 + b_ixy + c_iy^2 + d_ix + e_iy + f_i < 0, i=1, \dots , k\},\] 
where $(a_i,b_i,c_i,d_i,e_i,f_i)\neq 0$, and \( b_i^2- 4a_i c_i = 0\).

We define a parabolic function as a convex quadratic function with a zero discriminant. We write $q_p(x,y) = ax^2+b x y+c y^2+ d x+e y +f \in\R^2$, where $a>0$, $c>0$, $b, d, e, f\in\R$  and  \(b^2 - 4ac = 0\) .

Finally, we use standard definitions for subdifferentials, subgradients, and normal cones; see for example~\cite{BAUSCHKE-24,HIRIART-URRUTY-93a,HIRIART-URRUTY-93b,ROCKAFELLAR-70a}.

While the convex envelope of a piecewise function is not always the same as the convex envelope of each piece, for bivariate quadratic polynomials we can leverage the convex envelope of each piece to obtain the convex envelope. Let us denote the $i^{th}$ piece of a piecewise function $f$ by $(f_i,P_i)$ where $f_i$ is a function and $P_i$ is a set.
From \cite[Theorem 3.4.1]{HIRIART-URRUTY-93b}, we have $(\inf_i f_i)^*=\sup_i f_i^*$. Using \cite[Proposition 2.6.1]{HIRIART-URRUTY-93b} and the biconjugate theorem $\co f = f^{**}$, we get
\begin{align*}
	\co[\inf_i(f_i+I_{P_i})]=&\co[\inf_i(\co(f_i+I_{P_i}))],\\
	=&[\inf_i(\co(f_i+I_{P_i}))]^{**},\\
	=&[\sup_i([\co(f_i+I_{P_i})]^{*})]^*.
\end{align*}

Hence we can find the biconjugate of a bivariate PLQ function by the following steps:
\begin{enumerate} \label{steps}
	\item Compute the convex envelope of each piece $\co(f_i+I_{P_i})$.
	\item Compute the conjugate of each convex piece $[\co(f_i+I_{P_i})]^{*}$.
	\item Compute the maximum of the conjugates over the entire PLQ function to obtain
	$f^*=\sup\limits_{i}(\co(f_i+I_{P_i})^{*})$.
    \item Compute the biconjugate by computing the conjugate of the maximum of the conjugates over the entire PLQ function $f^{**}=(\sup\limits_{i}(\co(f_i+I_{P_i})^{*}))^*$.
\end{enumerate}

Step 1 of this process, finding the convex envelope of each piece $(f,P)$ with $f(x)=x^TAx+b^Tx+c$, is computed as follow. We compute the eigenvalues of $A$. If they are both nonnegative, the function is convex and we are done. If they are both nonpositive, the function is concave and the convex envelope is obtained as the convex hull of the points $\{(x,f(x)) : x$ is a vertex of $P\}$. If $A$ is indefinite (one positive and one negative eigenvalue), we note that $\co f(x) = \co (x^TAx) + b^Tx+c$; see~\cite[p. 93]{HIRIART-URRUTY-93a}. Hence, we only focus on the quadratic part. We rotate the polyhedral set so that in the new basis, the function is now $(x_1,x_2)\mapsto x_1 \cdot x_2$. 

We consider triangulations of the domain. When we restrict the bilinear function $(x_1,x_2)\mapsto x_1 \cdot x_2$ to and edge and get a strictly convex function, we call that edge a convex edge. 
Each piece $(P_i, f_i)$ of the domain now consists of triangle $P_i$ with either zero, one, two, or three convex edges. To avoid cases with three convex edges, we split such triangles into two subtriangles by adding a horizontal edge through a vertex such that each triangle has two strictly convex edges. Hence, we will only consider triangles with zero, one or two convex edges.

We then use direct formulae, which we developed using the method given in~\cite{LOCATELLI-16}, to obtain the convex envelopes of bilinear forms over triangles. The derivation of those formulae is explained in Appendix~\ref{s:convexenvelope_formulae}.

Kumar~\cite{KUMAR-19a} showed how to compute the conjugate of a rational function defined on a polytope. We further fine tuned his method to obtain formulae given in Appendix~\ref{s:conjugate_formulae} to compute the conjugate which is Step 2 of the above process. 
We then computed the maximum of the functions obtained in Step 3 thereby obtaining the conjugate of a bivariate PLQ function.

Beyond simplifying previous work~\cite{LOCATELLI-16,KUMAR-19a}, our focus here is to explain how to perform the last step, required to obtain the convex envelope, by computing the conjugate of the conjugate obtained in Step 3 thereby obtaining the biconjugate of a bivariate PLQ function. Figure~\ref{fig:process} shows the overall process of computing the biconjugate. As shown in the figure, we can now skip Step 1 as we have direct formulae to get the conjugate of the bilinear form on a triangle. 

\begin{figure}
        \centering
        \includegraphics[width=.7\textwidth]{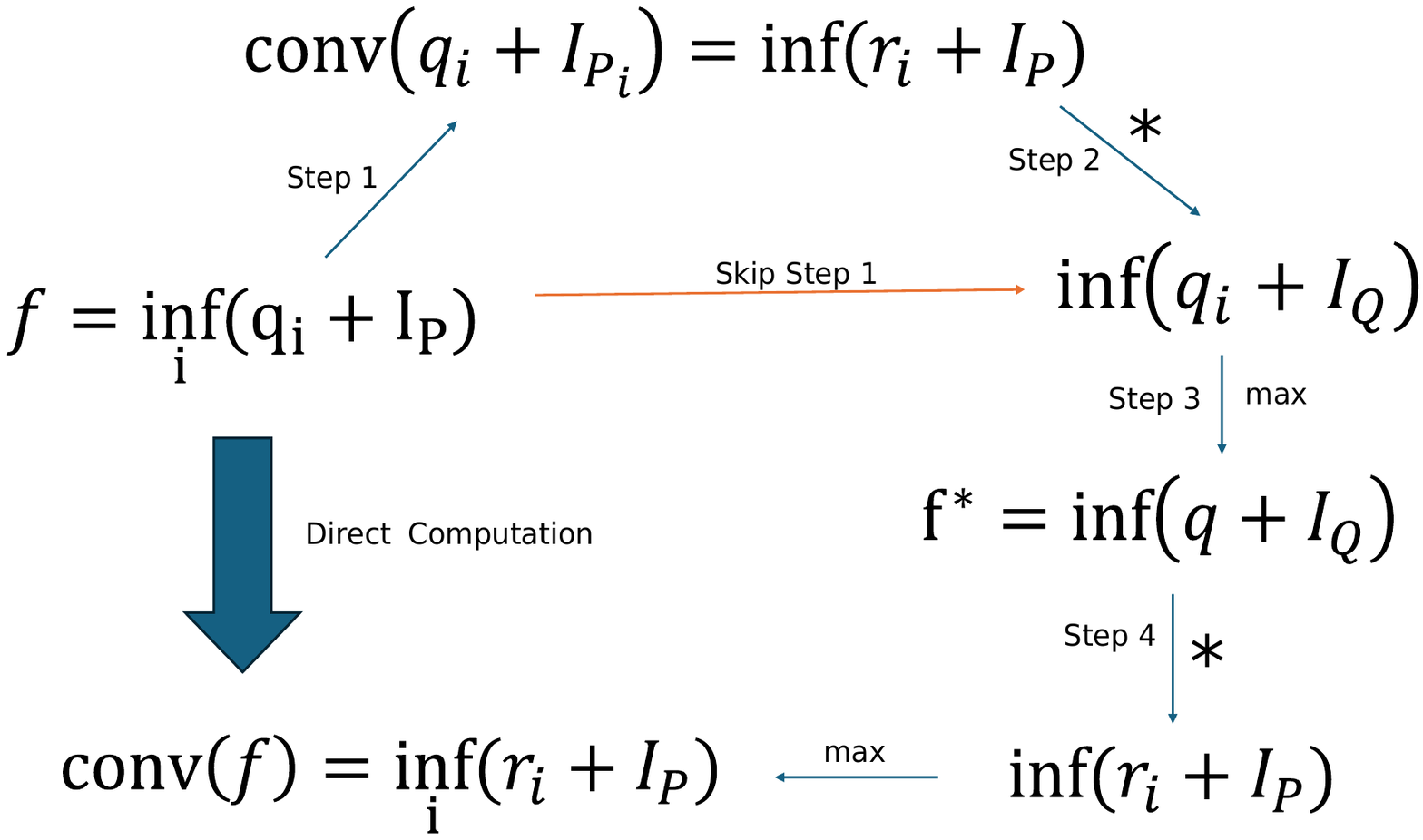}
        \caption{Description of the process. The computation follows the thin arrows steps 1, 2, 3, 4 and max. Future work will look at direct computation (thick arrow).}
        \centering
        \label{fig:process}
    \end{figure}

\FloatBarrier

Our implementation in MATLAB uses symbolic computation and rational numbers to avoid floating-point errors. Using floating-point arithmetic can give rise to partitions of the domain with degenerate subsets, and floating-point coefficients of functions make it difficult to detect equal functions on adjacent domains. Indeed, one of the motivations of~\cite{SINGH-21} is to check that intermediate computations of conjugates are still convex (as they should be without floating point errors). After observing these difficulties, we decided to work with symbolic representation of our functions with rational coefficients.

\section{Structure of the conjugate of a PLQ function}\label{s:struct}
    The conjugate of a PLQ function has a particular form. It is a piecewise function with four types of pieces:
    \begin{enumerate}
        \item Type 1 : Linear functions on a polyhedral subdivision, \eg pieces 2,3,4,6 and 7 from Figure~\ref{fig:ex_max}.
        \item Type 2 : Linear function defined on a region with a parabolic edge and two parallel linear edges. This region is the convex side of the parabolic edge (\eg Piece 1 of Figure~\ref{fig:ex_max}).
        \item Type 3 : Parabolic function defined on the nonconvex side of the parabolic edge between the same two linear parallel edges(\eg Piece 5 of Figure~\ref{fig:ex_max}). Note that this region is nonconvex because it obtained as the union of convex sets (rays) by merging convex sets to simplify the function representation.
        \item Type 4 : Parabolic function on an unbounded polyhedral subdivision which has three edges, two of which are parallel (\eg Figure~\ref{fig:quadPoly}).
    \end{enumerate}
    
    We exploit this structure to compute the biconjugate of the PLQ function.

    \begin{figure}
        \centering
        \def\svgwidth{0.40\textwidth}
        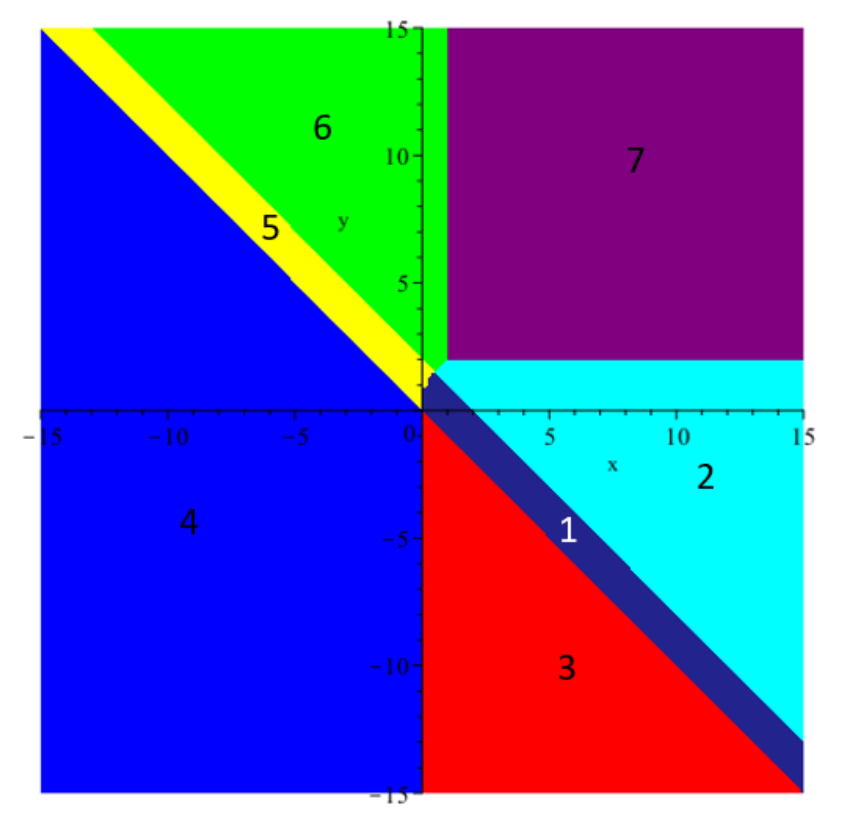
       \caption{Piecewise function  (Parabolic functions defined on a polyhedral subdivision.) (Functions are defined in Table~\ref{table:Fig1})}
        \label{fig:ex_max}
   \end{figure}

    \begin{table}[tbph]
\centering
\caption{Function definitions for Figure \ref{fig:ex_max}.}
	\label{table:Fig1}
\begin{tabular}{@{}llr@{}} \toprule 
    Piece Number & Expression\\
    \midrule
	1 &$2s_1$ \\
    2 &$2s_1$ \\
    3 &$2s_1$ \\
    4 &$(s_1+s_2)^2/4$ \\
    5 &$s_1+s_2-1$ \\
    6 &$2s_1+s_2-2$ \\
\end{tabular}
\end{table}

   \begin{figure}%
	\centering
	\subfigure[ Type 4   ]{
		\centering
        \includegraphics[width=.45\textwidth]{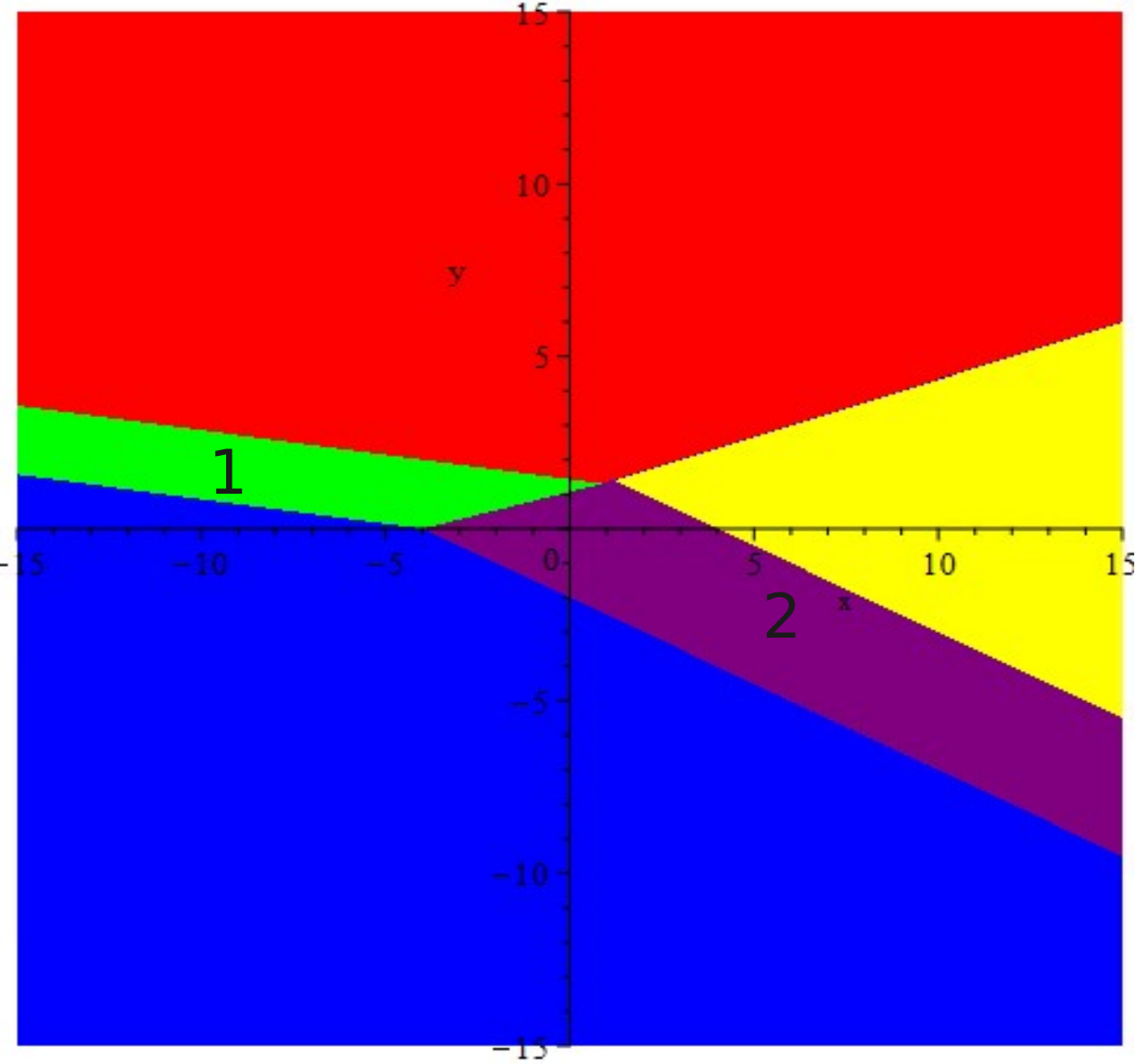}}	
	\subfigure[ Subdifferential of Piece 2  ]{
		\centering
	  \includegraphics[width=.45\textwidth]{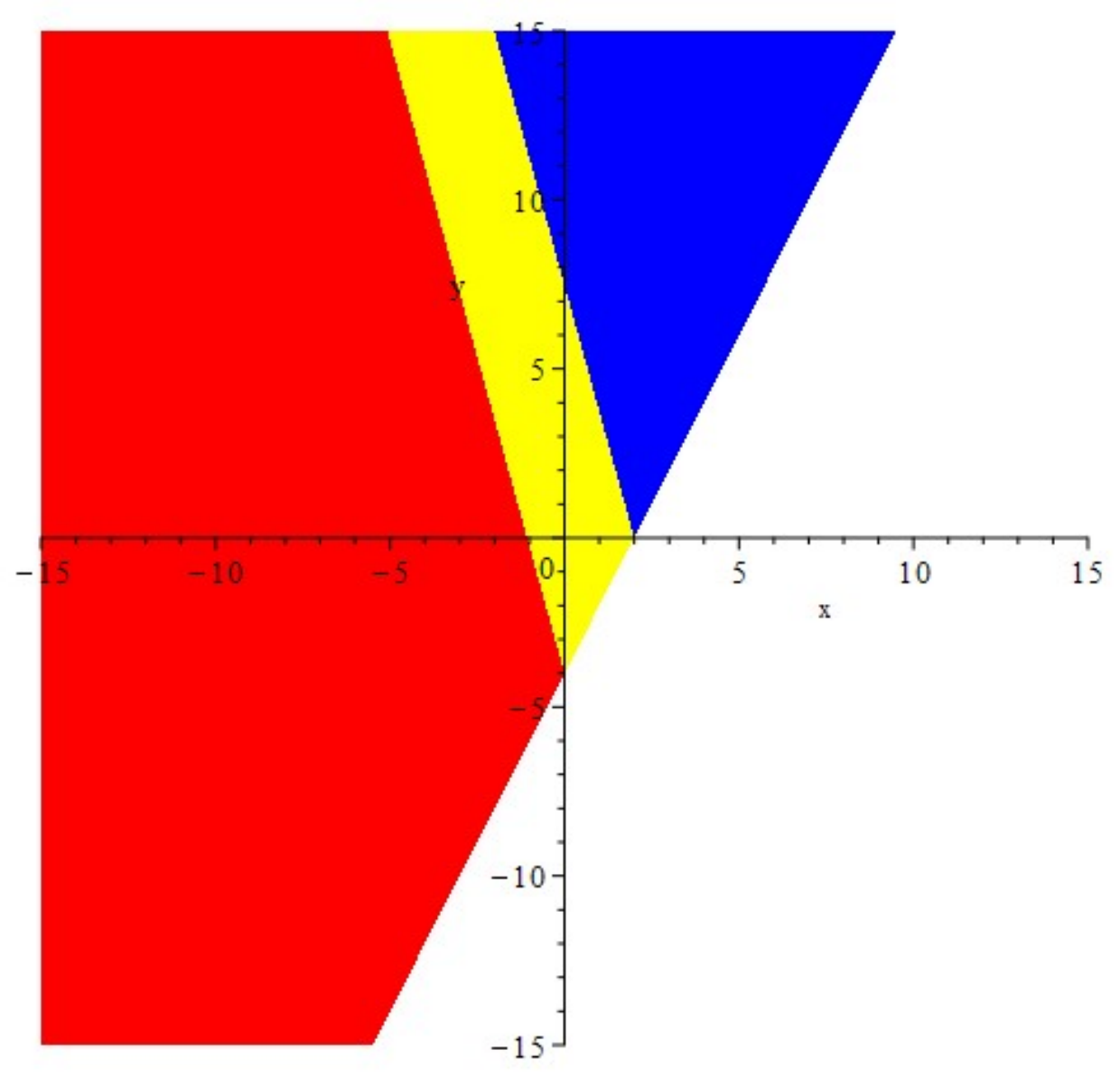}}	\\
	\caption{Piecewise function  (Parabolic functions defined on a polyhedral subdivision.)(Functions are defined in Table~\ref{table:Fig2})}
	\label{fig:quadPoly}
\end{figure}

\begin{table}[tbph]
\centering
\caption{Function definitions for Figure \ref{fig:quadPoly}.}
	\label{table:Fig2}
\begin{tabular}{@{}llr@{}} \toprule 
    Piece Number & Expression\\
    \midrule
	1 &$\frac{1}{8}s_1^2 + \frac{1}{2}s_1s_2 + s_1 + \frac{1}{2}s_2^2 - 2s_2 + 2$ \\
    2 &$\frac{1}{28}s_1^2 + \frac{1}{2}s_1s_2 + \frac{2}{7}s_1 + \frac{7}{4}s_2^2 - 2s_2 + \frac{4}{7}$ \\
    
\end{tabular}
\end{table}

   \FloatBarrier

We now discuss the computation of the conjugate of a parabolic function defined over a parabolic subdivision. We start with computing the subdifferentials corresponding to each entity of the subdivision and then compute the expressions over them to finally obtain the conjugate. An entity of the parabolic subdivision is either a vertex, an edge or a face of the parabolic subdivision of $\dom f^*$

\section{Domain of the biconjugate}\label{s:dom}
    Given a nonconvex PLQ function, we compute the conjugate and obtain a piecewise parabolic function.  We now compute the conjugate of each piece of the conjugate by first computing its domain which turns out to be a polyhedral subdivision. (Recall that $\overline{\dom} f^* = \overline{\ran} \partial f$ %
    \cite[Remark 18.16]{BAUSCHKE-24}.) We decompose the domain of $f^*$ (parabolic subdivision) into its interior, its vertices and its edges. 
    
    The subdifferential of any point in the interior or any of the vertices is identical for all four types of pieces mentioned in Section~\ref{s:struct}.

    Let 
    \begin{equation} \label{eq:qp}
     q_p(x) = ax_1^2+bx_1x_2+cx_2^2+dx_1+ex_2+f, \text{ with }  a>0, c>0, b^2-4ac=0, 
    \end{equation}
    be a parabolic function. Note that $q_p$ is a convex function (the letter q is used for quadratic, and the subscript p for parabolic).

    \begin{proposition}[Interior] \label{Prop-1}
        Consider $q_p$ defined in~\eqref{eq:qp}. Then there exists $\alpha_i$ such that 
        \[\bigcup_{x\in\dom q_p} \partial q_p(x) \subseteq \{s:L(s)=0\},\] 
        where $L(s)=\alpha_1s_1+\alpha_2s_2+\alpha_0$. 
    \end{proposition}

    \begin{proof}
        Take $s\in \bigcup_{x\in\dom q} \partial q(x)$. There exists $x$ such that $s_1=2ax_1+bx_2+d$ and $s_2=bx_1+2cx_2+e$. Define $\alpha_1=b$, $\alpha_2=-2a$, $\alpha_0=-db+2ae$, and substitute values of $s_1$ and $s_2$ to obtain
        \[\alpha_1s_1+\alpha_2s_2+\alpha_0 = (b^2-4ac)x_2.\]
        Using $b^2-4ac=0$ finishes the proof.
    \end{proof}
    
    \begin{remark}
    Proposition~\ref{Prop-1} is more general than~\cite[Proposition 1]{KUMAR-19} as a polyhedral subdivision is a special case of a parabolic subdivision.     
    \end{remark}
    
    \begin{corollary}[Interior] \label{Cor-1}
        For a bivariate parabolic function $q_p$ as defined in~\eqref{eq:qp}, and a set $P$ belonging to a parabolic subdivision $\PR$, define $f(x) = q_p(x)+I_{P}(x)$, then for all $x \in \inte(P)$, the set $\cup_{x \in \inte(P)} \partial f(x)$ is contained in a line.
    \end{corollary}

    \begin{proof}
        For any point $x\in \inte P$, we have $\partial f = \partial q$ so the result follows by Proposition~\ref{Prop-1}.
    \end{proof}

    \begin{example}
        For the parabolic function, $q(x)=(x_1+x_2)^2/4$ defined in the interior of a parabolic division as shown in Figure~\ref{fig:ex_max}, the subdifferential is contained in a line segment $\{s:s_2=0\}$ between $(0,0)$ and $(1,1)$ as shown in Figure \ref{fig:exp}[e].    
    \end{example}

    \begin{proposition}[Vertices] \label{Prop-2}
        For a parabolic function $q_p$ as defined in~\eqref{eq:qp}, a parabolic unbounded region $P$, and a vertex $v$ of $P$. Let $f(x) = q_p(x) + I_{P} (x)$. Then $\partial f(v)$ is an unbounded polyhedral set.
    \end{proposition}

    \begin{proof}
        Let $v$ be a vertex of $P$ and $E^-$ and $E^+$ be the tangents to the edges at the vertex $v$. Then for all $y \in P$, $y_l \in E^-$ and $y_r \in E^+$, the normal cone at $v$ is 
        \begin{equation}
            \begin{split}
                N_{P}(v) & = \{s \in \R^2 : s^T(y-v) \le 0, \forall y \in P \}, \\  
                          & = \{s_l \in \R^2 : s_l^T(y_l-v) \le 0 \} \cap \{s_r \in \R^2 : s_r^T(y_r-v) \le 0 \}, \\      
                          & = \{s \in \R^2 : s^T(y_l-v) \le 0,  s^T(y_r-v) \le 0 \}, \\      
            \end{split}        
        \end{equation}
        which is an unbounded polyhedral set. Therefore the subdifferential given by
        \begin{equation}
            \begin{split}
                \partial f(v) & = \partial q_p(v) + N_{P}(v), \\  
                              & = \{s + \nabla q_p(v) \in \R^2 : s^T(y_l-v) \le 0,  s^T(y_r-v) \le 0 \}, \\      
            \end{split}        
        \end{equation}
        is also an unbounded polyhedral set. 
    \end{proof}

    \begin{remark}
        The polyhedrality of the normal cone can also be deduced from \cite[Theorem 6.14]{ROCKAFELLAR-98a} and the subsequent note on p. 211 that states that the normal cone is generated by gradients of the active constraints. Since there are a finite number of inequalities and all functions are differentiable, there is a finite number of such gradients, which makes the normal cone polyhedral.
    \end{remark}

    \begin{remark}
    Since linear functions are a subset of parabolic functions, our result generalizes~\cite[Lemma 1]{KUMAR-19}.    
    \end{remark}
    
    \begin{example}
        Consider piece $P_1$ shown in Figure~\ref{fig:ex_max}
        \[ P_1 = \{(x,y) : -x-2y-4 \le 0, x+2y-4 \le 0,  48x - 56y + 4xy + x^2 + 4y^2 - 184 \le 0  \}. \]
        The function $f(x,y)=2x$ is defined on $P_1$. At the vertex $V=(0,0)$, $\nabla q(V) = (2,0)$ and the normal cone is 
        $ N_{P_1}(V)= \{ s : s_2 \le 0, s_1-s_2\le2 \}$ shown in Figure~\ref{fig:piece1SubdiffV1}. The subdifferential is given by $\partial f_{P_1}(V) = \{(2,0) + (s_1,s_2), s_2 \le 0, s_1-s_2\le2\}$.
    \end{example}

    \begin{figure}
        \centering
                 \centering
                 \includegraphics[width=.45\textwidth,trim=150 100 100 350, clip]{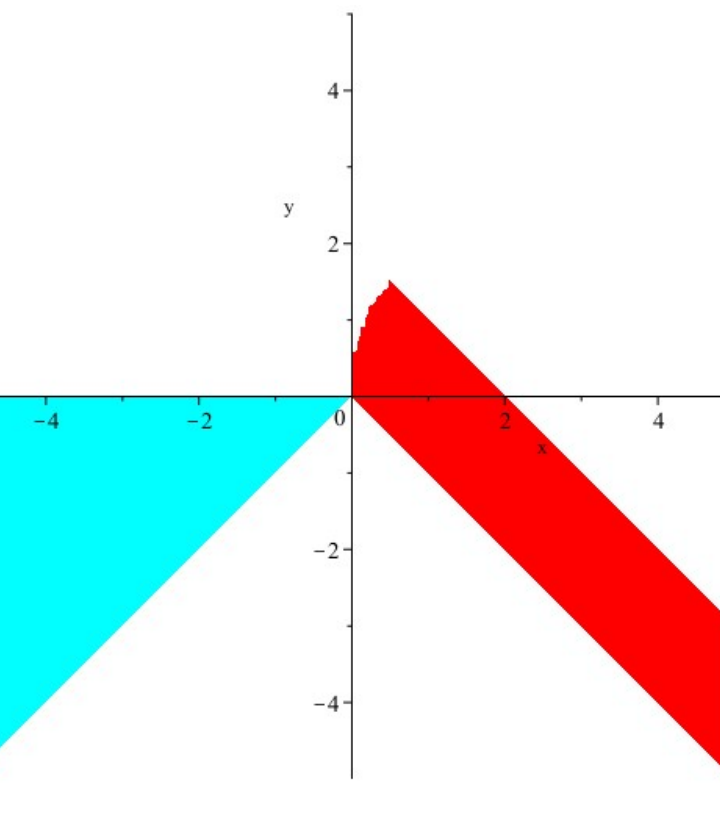}
        \caption{Piece 1 - Vertex 1. Normal Cone at Vertex $v=(0,0)$.}
        \label{fig:piece1SubdiffV1}
    \end{figure}

    Subdifferentials computed at the vertices of a polyhedral regions are given in \cite{KUMAR-19} and can also be considered as a particular case of Proposition~\ref{Prop-2} as a polyhedral division is a special type of parabolic division.
    
    For computing the subdifferentials over the edges, we consider all the cases listed in Table~\ref{table:edges}.

    \begin{table}[tbph]
        \centering
        \caption{Subdiffferential based on edges.
        \label{table:edges}}
        \begin{tabular}{@{}cllcl@{}} \toprule 
        
            Type of Piece & Function & Edge & Lemma & Subdifferential\\\toprule 
            1, 2 & Linear & Linear & \ref{Lemma-1} & Singleton\\ 
            2 & Linear & Parabolic & \ref{Prop-4} & Cone\\ 
            3, 4 & Parabolic & Linear(Parallel) & \ref{Cor-2} & Ray\\ 
            3 & Parabolic & Parabolic & \ref{Cor-3} & Polyhedral region\\ 
            4 & Parabolic & Linear & \ref{Cor-3} & Polyhedral region\\             
             \bottomrule 
        \end{tabular}
    \end{table}

    \begin{lemma}[Linear edge and a linear {function~\cite[Corollary 4.5]{KUMAR-19a}}] \label{Lemma-1}
        Let $l$ be a linear function defined on a linear edge $E = \{x:x_2=mx_1+c,x_{l_1} \le x_1 \le x_{u_1} \}$ ($x_{l_1}$ or $x_{u_1}$ maybe $\infty$). Then $\partial l(E)$ is a singleton.
    \end{lemma}

    \begin{proposition}[Linear edge and a parabolic function. ] \label{Prop-3}
        Take $q_p$ as in \eqref{eq:qp},  $P\in\PR$ a parabolic region, and $L = \{x:x_2-mx_1-c=0,x_{l_1} \le x_1 \le x_{u_1}\}$ a linear edge. Then for any $x\in L$ and $f(x) = q_p(x) + I_{P} (x)$, $\partial f(x) = \{ \Bar{s} + \hat{s}:  b\Bar{s_1}-2a\Bar{s_2}- bd+ 2ae=0, \hat{s_1}=-m\hat{s_2}, \hat{s_2} \ge 0\}$ is a polyhedral region or a ray.
    \end{proposition} 
            
    \begin{proof}
        For $x \in  \ri (L)$, 
        \begin{equation}
            \begin{split}
            N_P(x) %
                   =  &\{\hat{s}:\hat{s}=\lambda (-m,1)), \lambda \ge 0\},   \\
                   =  &\{\hat{s}:\hat{s_1}=-ms_2, \hat{s_2} \ge 0\}.  
            \end{split}
        \end{equation}
       From Proposition~\ref{Prop-1} we have 
       \[\bigcup_{x\in\dom q_p} \partial q_p(x) \subseteq \{\Bar{s}:b\Bar{s_1}-2a\Bar{s_2}- bd+ 2ae=0\}.\]  

        Now for all $x \in  \ri (L)$, 
         $\partial f(x) = \partial q_p(x) + N_P(x)$, so
         \[ f(x) = \{ \Bar{s} + \hat{s}:  b\Bar{s_1}-2a\Bar{s_2}- bd+ 2ae=0, \hat{s_1}=-m\hat{s_2}, \hat{s_2} \ge 0\},\]
         which concludes the proof.
    \end{proof}
    
    \begin{corollary} \label{Cor-2}
        For a parabolic function $q_p$ restricted to a linear edge $L$ of a piece of Type 3 $P$, $\partial f(x)$ is contained in a ray.
    \end{corollary}

    \begin{proof}
        We get pieces of Type 3 by taking the conjugate of a rational function restricted to a convex edge. The linear edges are parallel to the normal cone obtained at the convex edge. The slope of a parallel edge of $f^*$ satisfies $m_1=-1/m$, where m is the slope of the convex edge of $f$ and $m_1$ is the slope of the edge L. The parabolic function computed for this type is given by $q_p = ax^2+bxy+cy^2+dx+ey+f$, where
        $a = 1/(4m)$,
        $b = 1/2$,
        $c = m/4$,
        $d = -q/(2m)$,
        $e = q/2$,
        $f = q^2/(4m)$; see Appendix ~\ref{s:conjugate_formulae}.
        Substituting the value for $m_1$, and the values for the coefficients in $\partial f(x)$ obtained in Proposition~\ref{Prop-3}, we get 
        \[\partial f(x) = \{ \Bar{s} + \hat{s}:  m\Bar{s_1}-\Bar{s_2} + q/2m=0, \hat{s_2}=m\hat{s_1}, \hat{s_2} \ge 0\}.\]     
        This gives us a ray.
    \end{proof}
    
    \begin{example}
        Consider piece $P_5$ shown in Figure~\ref{fig:ex_max}
        \[ P_5 = \{(x,y) : -x-2y-4 \le 0, x+2y-4 \le 0,  -48x + 56y - 4xy - x^2 - 4y^2 + 184 \le 0  \}. \]
        The function $f(x_1,x_2)=(1/4)(x_1+x_2)^2$ is defined on this region. At the edge, $x+y=0$ which is the ray from  vertex, $V_1(0,0)$, $\nabla f(x_1,x_2) = (1/2)(s_1 - s_2)$ and the subdifferential is contained in the ray $\{ s :  s_2 - s_1 = 0, s_2 \le 0\}$.
    \end{example}

    \begin{corollary} \label{Cor-3}
        For a parabolic function $q_p$ restricted to a linear edge (not parallel) of a piece of Type 4, $\partial f(x)$ is an unbounded polyhedral region.
    \end{corollary}

    \begin{example}
        Refering to the purple piece in Figure~\ref{fig:quadPoly}, corresponding to the edge between the purple and green piece we get the yellow region as shown in the subdifferential in Figure~\ref{fig:quadPoly}.
    \end{example}

    \begin{remark}
         The pieces of ``Type 3'' are nonconvex, joining the two vertices with a line gives us its convex envelope. We compute the normal cone of the convex envelope in order to compute the domain of the conjugate. We use Corollary~\ref{Cor-3} to compute the normal cone.
    \end{remark}

    \begin{proposition}[Parabolic edge and linear function] \label{Prop-4}
        For a linear function $l$, restricted to the parabolic edge $C(x) = \{x:ax_1^2+bx_1x_2+cx_2^2+dx_1+ex_2+f, b^2-4ac=0, x_{l_1} \le x_1 \le x_{u_1} $\}, then $\partial f(E)$ is a cone with vertex at $\nabla l(x)$.
    \end{proposition}
        
    \begin{proof}
    
        Let $g(x)$ be the restriction of a linear function to the parabolic curve joining $x_l$ and $x_u$ be given by
        \begin{equation}
             C(x)=ax_1^2+bx_1x_2+cx_2^2+dx_1+ex_2+f, b^2-4ac=0.  
             \label{Pcurve}
        \end{equation}
        For $s\in \bigcup_{x\in\dom g} \partial g(x)$, $s = \lambda \nabla C(x), \lambda \ge 0.$
        This gives us 
        \begin{equation}
            s_1 =  \lambda(2ax_1+bx_2+d),
            s_2 =  \lambda(bx_1+2cx_2+e).
            \label{grads}
        \end{equation}
        
        Solving equations \eqref{Pcurve} and \eqref{grads} for $x_1$, $x_2$ and $\lambda$ and substituting, we get the normal cone between rays, $s_2 = k_l s_1$ and $s_2 = k_r s_1$, where $k_l=(e + bx_l + 2cy_l)/(d + 2ax_l + by_l)$ and $k_r = (e + bx_r + 2cy_r)(d + 2ax_r + by_r)$.
        For all $x \in ri(C)$, $\partial f(x) = \partial g(x) + N_P(x)$, thus
        \[\partial f(x) = \{\nabla g(x) + (s_1,s_2), k_ls_1 \le s_2 \le k_rs_1\}.\]
    
    \end{proof}
    
    \begin{example}
        Consider piece $P_1$ shown in Figure~\ref{fig:ex_max}
        \[ P_1 = \{(x,y) : -x-2y-4 \le 0, x+2y-4 \le 0,  48x - 56y + 4xy + x^2 + 4y^2 - 184 \le 0  \}. \]
        The function $f(x,y)=2x$ is defined on $P_1$. At the edge between vertex $V_1(0,0)$ and vertex $V_2(0.5,1.5)$, $\nabla l(v) = (2,0)$ and the subdifferential is described by $-s_2 \le 0$ and $s_1+s_2-2\le0$ as shown in Figure~\ref{fig:SubdiffE12}.
    \end{example}

    \begin{figure}
        \centering
        \includegraphics[width=.45\textwidth,trim=100 200 100 200, clip]{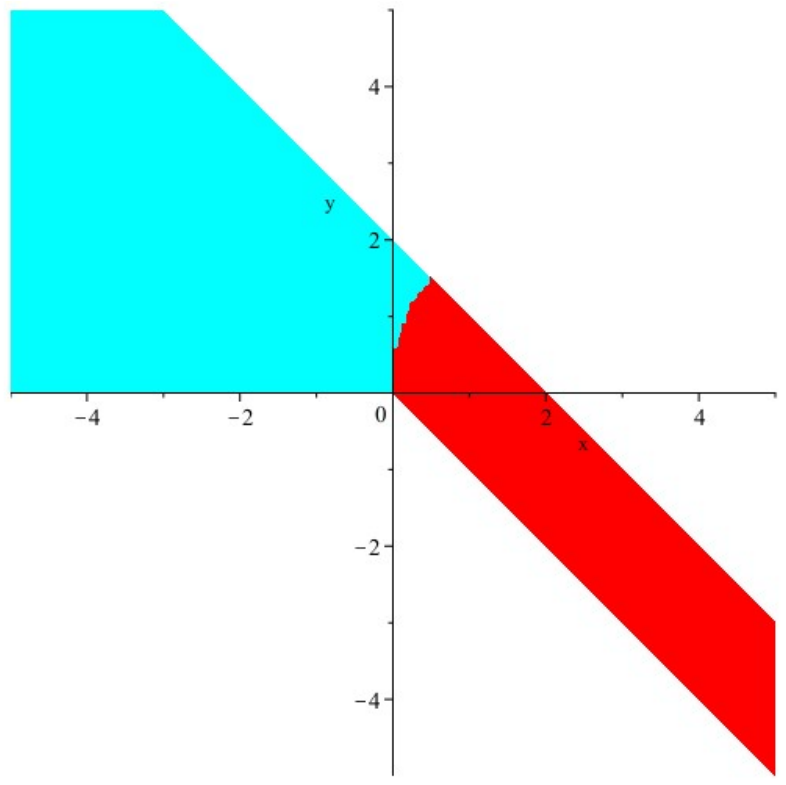}
        \caption{Subdifferential of $f=2x$ on the parabolic edge of $P_1$.}
        \centering
        \label{fig:SubdiffE12}
    \end{figure}

    Locatelli~\cite{LOCATELLI-16} showed that $\co f$ is a piecewise function over a polyhedral subdivision that can be linear, quadratic, or rational. By gathering Propositions~\ref{Prop-1}, \ref{Prop-2}, \ref{Prop-3} and  \ref{Prop-4}, Corollaries \ref{Cor-1}, \ref{Cor-2} and \ref{Cor-3} and Lemma  \ref{Lemma-1}  we get the same theorem with a constructive proof.
    
    \begin{theorem}[Polyhedral domain] \label{Thm-1}
        For a piecewise quadratic-linear function, the domain of the biconjugate has a polyhedral subdivision. 
    \end{theorem}

\section{Biconjugate Expression}\label{s:expr}
    Now that we know $\dom f^{**}$  as a polyhedral subdivision, we turn to the computation of its expression on each entity.
        
    The conjugate expressions we obtain are either linear functions 
     \begin{equation}\label{eq:linear}
     ax+by+c,
     \end{equation}
    parabolic functions 
     \begin{equation}\label{eq:parabolic}
    ax^2+bxy+cy^2+dx+ey+f, a,c > 0, b^2-4ac=0
    \end{equation}
    or a rational functions of the form $\xi_1(x)^2/\xi_2(x) + \xi_0(x)$ where $\xi_1$, $\xi_2$ and $\xi_0$ are linear functions. These rational functions only occur on an edge of the triangle where the function restricted to that edge is convex. (That edge cannot be vertical because then the function $f(x)=x_1\cdot x_2=0$.)
    Let the convex edge be $y=mx+q$ and the vertex opposite it be $(x_1,y_1)$, then the rational expression can be written as
    \begin{equation} \label{eq:rational}
    r(x,y)=\frac{ax^2+bxy+cy^2+dx+ey+f}{gx+hy+k},
    \end{equation}
    where $a=-my_1$, $b=q$, $c=x_1$, $d=-qy_1+mx_1y_1$, $e=-qx_1-x_1y_1$, $f=qx_1y_1$, $g=-m$, $h=1$, and $k=-y_1+ m x_1$.
    
     Note that \eqref{eq:parabolic} and \eqref{eq:linear} are special cases of \eqref{eq:rational}.

    \begin{table}[tbph]
            \centering
            \caption{Conjugate expressions and domain types for each piece type with associated cases as indicated in the proof of Theorem~\ref{t:bi}.
            \label{table:expr}}
            \begin{tabular}{c|l|l|l|l|c} \toprule 
                \bottomrule 
                Type of Piece & Function & Vertex & Domain & Expression & Case \\\toprule 
                1, 2, 3, 4 & Linear/Parabolic & Linear & Polyhedral region & Linear &  1\\ 
                \bottomrule 
                \bottomrule 
                Type of Piece & Function & Edge & Domain & Expression & Case \\\toprule 
                2 & Linear & Parabolic & Cone & Rational function &  3\\ 
                3 & Parabolic & Parabolic &Polyhedral region & Rational function &  2\\ 
                4 & Parabolic & Linear & Polyhedral region & Parabolic function &  4\\

             \bottomrule 
             
        \end{tabular}
    \end{table}

    \begin{theorem}\label{t:bi}
        For a PLQ function, the biconjugate has a polyhedral subdivision with a rational function of the form given in \eqref{eq:rational}.
    \end{theorem}

    \begin{proof}
        We are computing the expressions on nonempty regions as listed in Table~\ref{table:expr} and then computing the maximum over all the regions.
        \begin{case}[Vertices]
            Let $V \subset P$ be the set of all vertices. Then for $v \in V$
            \begin{equation}
                \begin{split}
                    f^*(s) = & \max_{x\in V}{(s^Tx-f(x))},\\
                           = & \max_{x\in V}{(s^Tx-(q(x)+I_{P}(x)))},\\
                           = & s^Tv-q(v),\\
                           = & s_1v_x+s_2v_y-q(v),\\
                \end{split}
                \label{eq:vexpr}
            \end{equation}
            is a linear function. So for any vertex $v$, the conjugate is a linear function defined on an unbounded polyhedral set.
        \end{case}

        \begin{case}[Edge:parabolic; Function:parabolic]
            
            When the edge is parabolic it is of a particular form; $edge = ax_1^2+bx_1x_2+cx_2^2+dx_1+ex_2+f,$ 
            where $a=-1$, $b=-2m$, $d=2q+4mx_1$, $c=-m^2$, $e=-(2mq - 4my_1)$ and $f=-(q^2 + 4mx_1y_1)$. 
            
            The function is $q_p = ax_1^2+bx_1x_2+cx_2^2+dx_1+ex_2+f,$ as given in Appendix \ref{s:conjugate_formulae}.
    
            As the conjugate expression is defined as $f^*(s)=sup_{x\in E}{(s^Tx-f(x))}$ and the function $q_p$ and $edge$ are differentiable, 
            we can compute the supremum to get
            \begin{equation} \label{eq:cexpr}
              f^*(s) = \frac{as_1^2+bs_1s_2+cs_2^2+ds_1+es_2+f}{gs_1+hs_2+k},  
            \end{equation}
            where $a=-my_1$,
            $b=q$,
            $c=x_1$,
            $d=-qy_1+mx_1y_1$,
            $e=-qx_1-x_1y_1$,
            $f=qx_1y_1$,
            $g=-m$,
            $h=1$, and
            $k=-y_1+ mx_1.$
    
        \end{case}

        \begin{case}[Edge:parabolic; Function:linear] 
            In this case, we get the conjugate expression as a rational function given in \eqref{eq:cexpr}. 
        \end{case}
       
        \begin{case}[Edge:linear; Function:parabolic] 
            When the edge is linear, $y=mx+q$ and the function is quadratic, $q=ax^2+bxy+cy^2+dx+ey+f$, we get the conjugate expression as a parabolic function.                
        \end{case}                    
    \end{proof}

    All the cases are verified symbolically and the code used to compute the conjugate expressions is given in Appendix \ref{s:conjugate_expr}.

        Computing the maximum over all the pieces, we obtain a rational function of the forms \eqref{eq:rational} on a polyhedral subdivision.

        This is illustrated in the examples that follow.

\section{Examples} \label{s:examples}
    In this section a few examples of the conjugates and biconjugates of PLQ functions are given. Example \ref{ex1} is a simple example with seven pieces which gives us a biconjugate with two pieces. Example~\ref{ex2} has more pieces. %

\begin{example}\label{ex1}

        The conjugate shown in Figure~\ref{fig:ex_max} has seven pieces, three vertices, one quadratic edge between the first two vertices, one linear edge (line) and some rays. We first compute the conjugate of each piece. The image by the subdifferential of vertices and the quadratic edge are polyhedral sets while the image of other pieces are contained in rays. It can be observed that a particular vertex (edge) from different pieces gives us the same conjugate function. 
        
        The conjugate of the first piece is the function $(s_1,s_2)\mapsto 2s_1$ defined over the domain given by the inequalities $- s_1 - s_2 \le 0 $, $2s_1s_2 - 8s_1 + s_1^2 + s_2^2 \le 0$ and $s_1 + s_2 - 2 \le 0$, and vertices $(0,0)$ and $(1/2,3/2)$. It is shown in Figure~\ref{fig:exp}[a].
                
        For the second piece, we obtain the function $s\mapsto 2s_1$ defined over the domain given by the inequalities $s_2 - 2 \le 0 $, $s_2 - s_1 - 1 \le 0$ and $2 - s_2 - s_1 \le 0$ and vertices $(1/2,3/2)$ and $(1/2,3/2)$; see Figure~\ref{fig:exp}[b].
       
        For the third piece function, we obtain $s\mapsto 2s_1$ defined over the domain given by the inequalities $-s_1 \le 0 $ and $s_1 + s_2 \le 0$ and vertex $(0,0)$; see Figure~\ref{fig:exp}[c].
       
        For the fourth piece, we obtain $s\mapsto 0$ defined over the domain given by the inequalities $s_1 \le 0 $ and $s_1 + s_2 \le 0$ and vertex $(0,0)$; see Figure~\ref{fig:exp}[d].
        
        For the fifth piece, we obtain $s\mapsto (1/4)s_1^2+(1/2)s_1 s_2+(1/4)s_2^2$ defined over the domain given by the inequalities $- s_1 - s_2 \le 0 $, $8s_1 - 2s_1s_2 - s_1^2 - s_2^2 \le 0$ and $s_1 + s_2 -2 \le 0$ and vertices $(0,0)$ and $(1/2,3/2)$; see Figure~\ref{fig:exp}[e].
        
        For the sixth piece, we get $s\mapsto s_1+s_2-1$ defined over the domain given by the inequalities $s_1 - 1\le 0 $, $s_1 - s_2 + 1 \le 0$ and $2 - s_2 - s_1 \le 0$ and vertices $(1,2)$ and $(1/2,3/2)$; see Figure~\ref{fig:exp}[f].

        For the last piece, we get $s\mapsto 2s_1+s_2-2$ defined over the domain given by the inequalities $1 - s_1 \le 0 $ and $2 - s_2 \le 0$ and vertex $(0,0)$. The subdifferential is shown in Figure~\ref{fig:exp}[g].

       \begin{figure}%
        	\centering
        	\subfigure[ $f_1^{**}$   ]{
        		\centering
        		\includegraphics[height=80px]{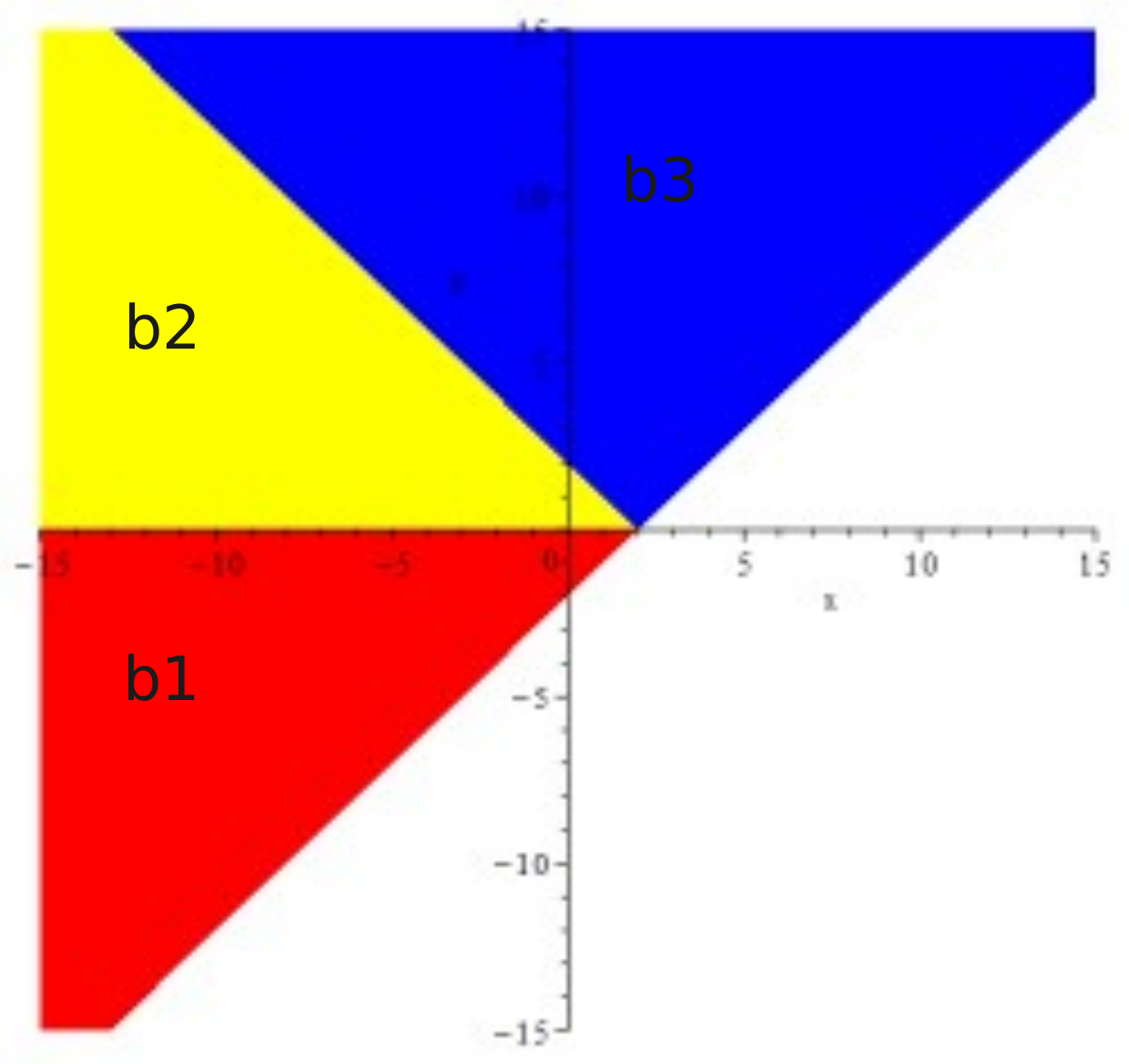}}	
        	\subfigure[ $f_2^{**}$   ]{
        		\centering
        		\includegraphics[height=80px]{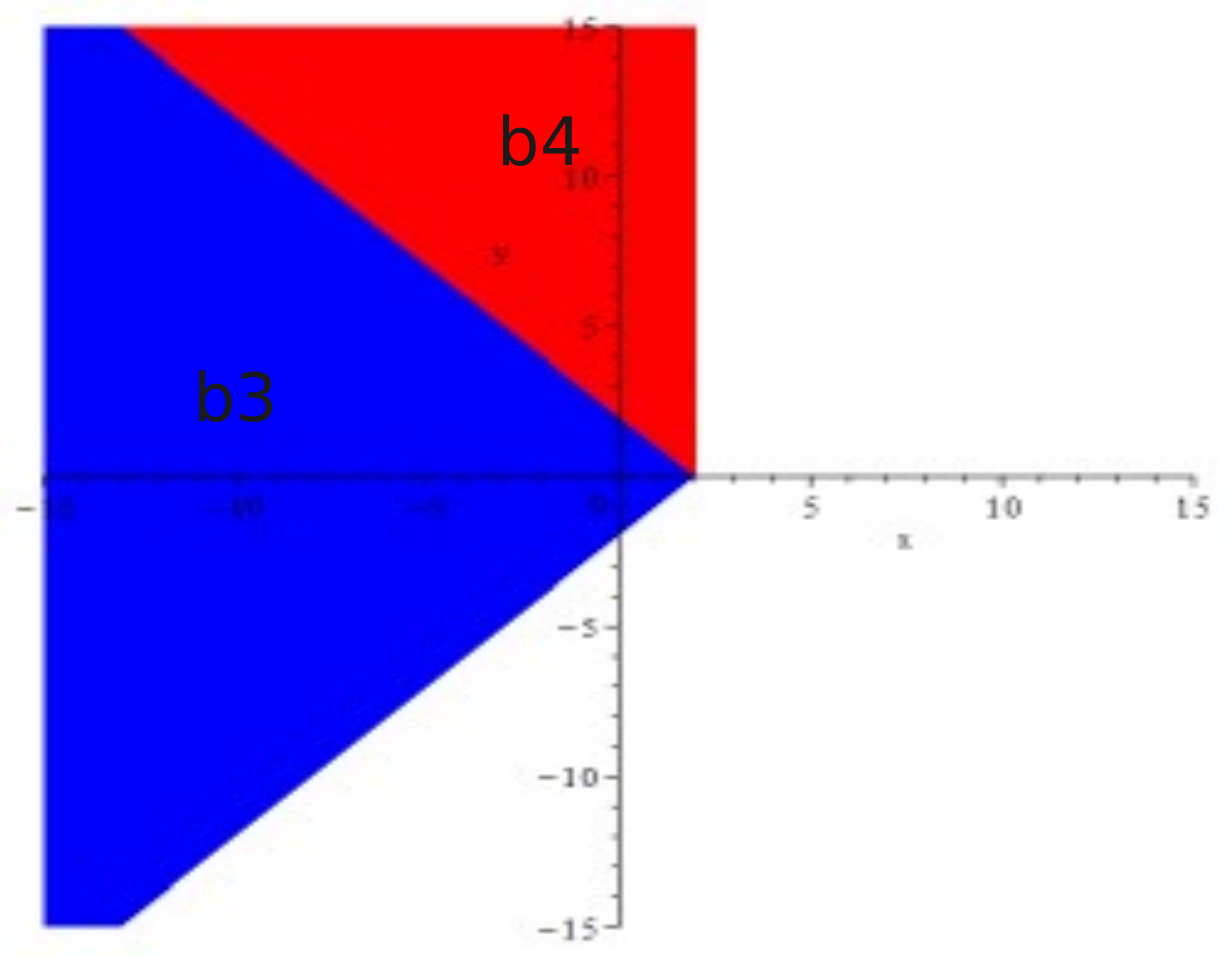}}	
        	\subfigure[ $f_3^{**}$   ]{
        		\centering
        		\includegraphics[height=80px]{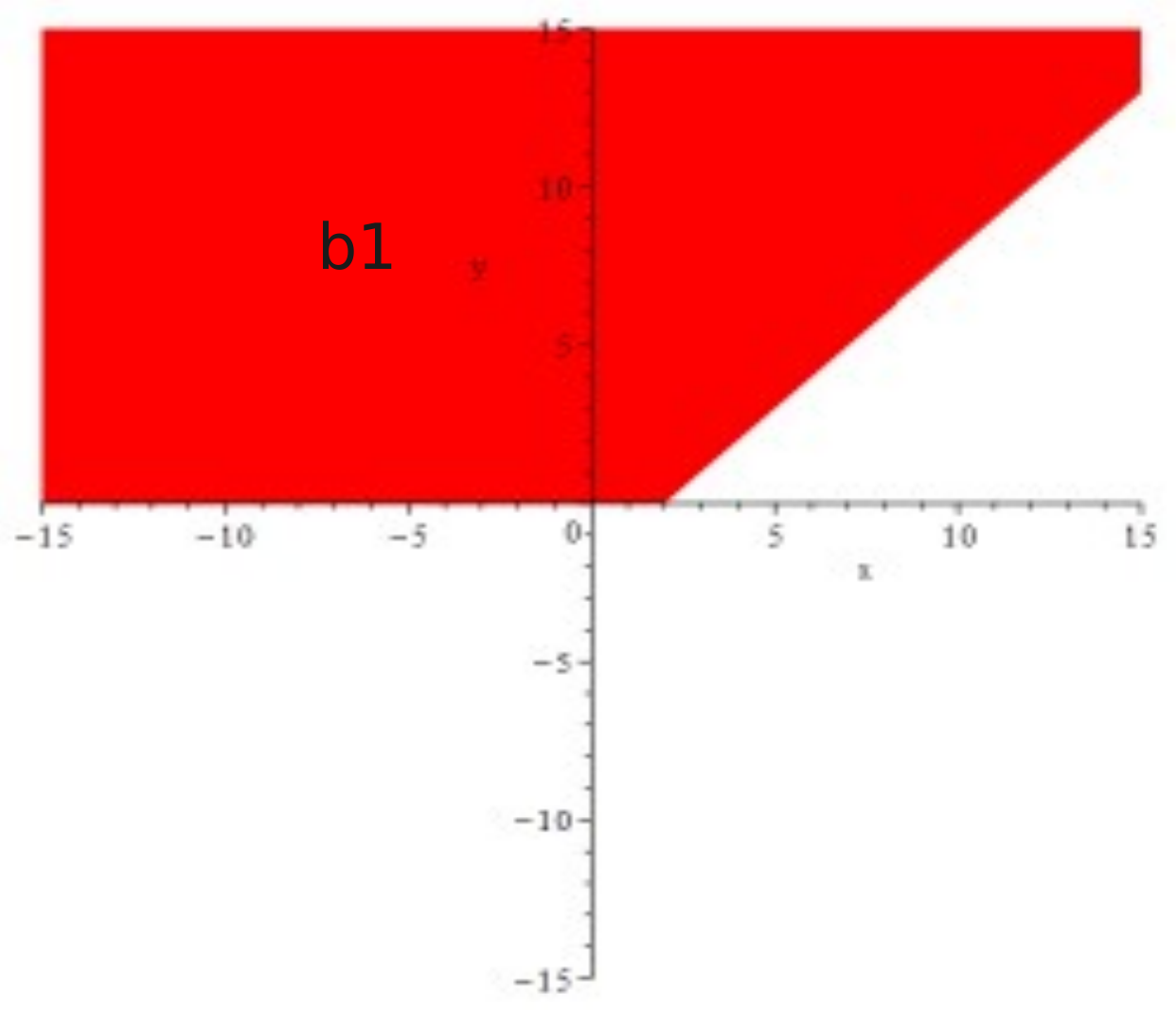}}	
        	\subfigure[ $f_4^{**}$   ]{
        		\centering
        		\includegraphics[height=80px]{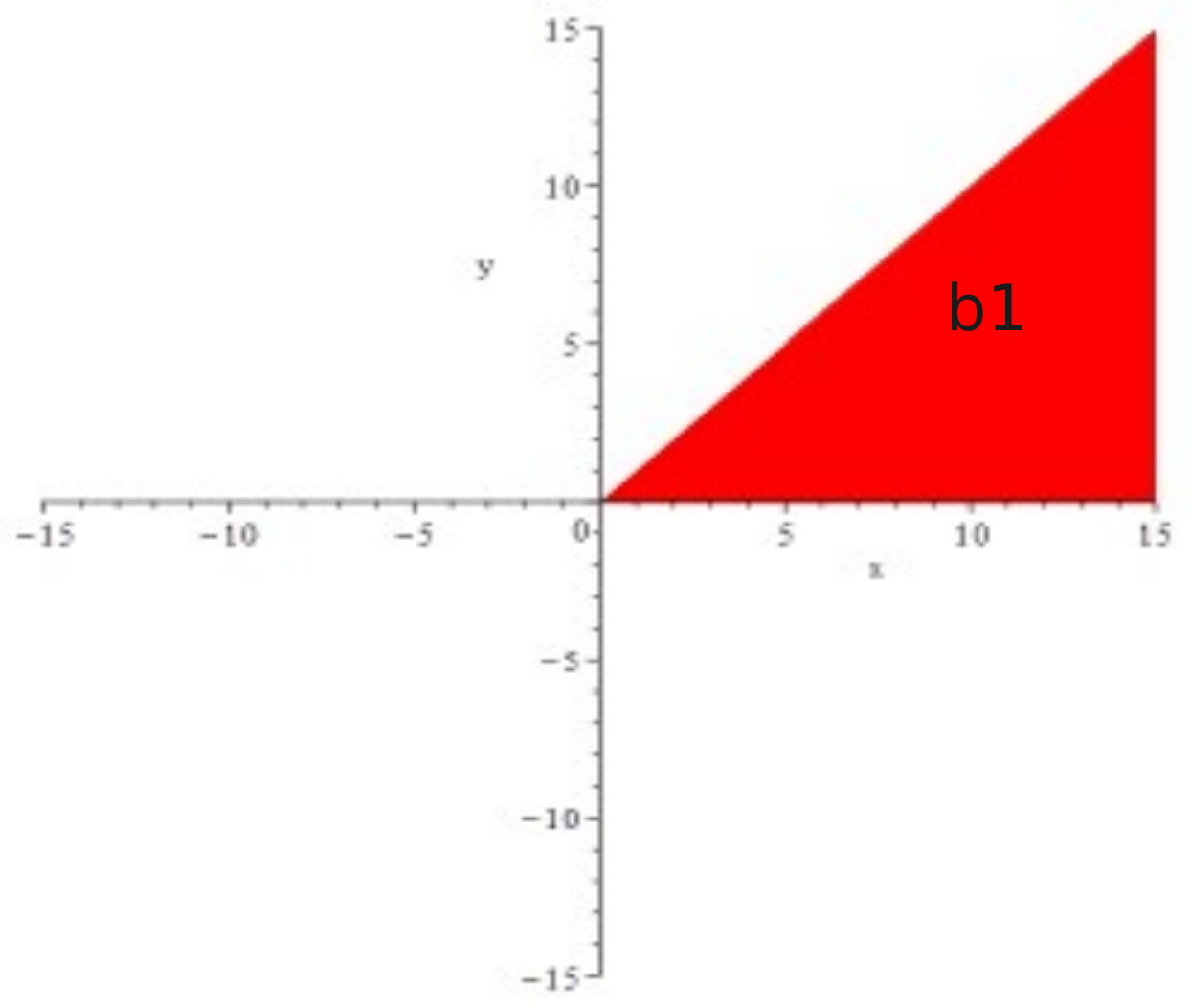}}	
        	\subfigure[ $f_5^{**}$   ]{
        		\centering
        		\includegraphics[height=80px]{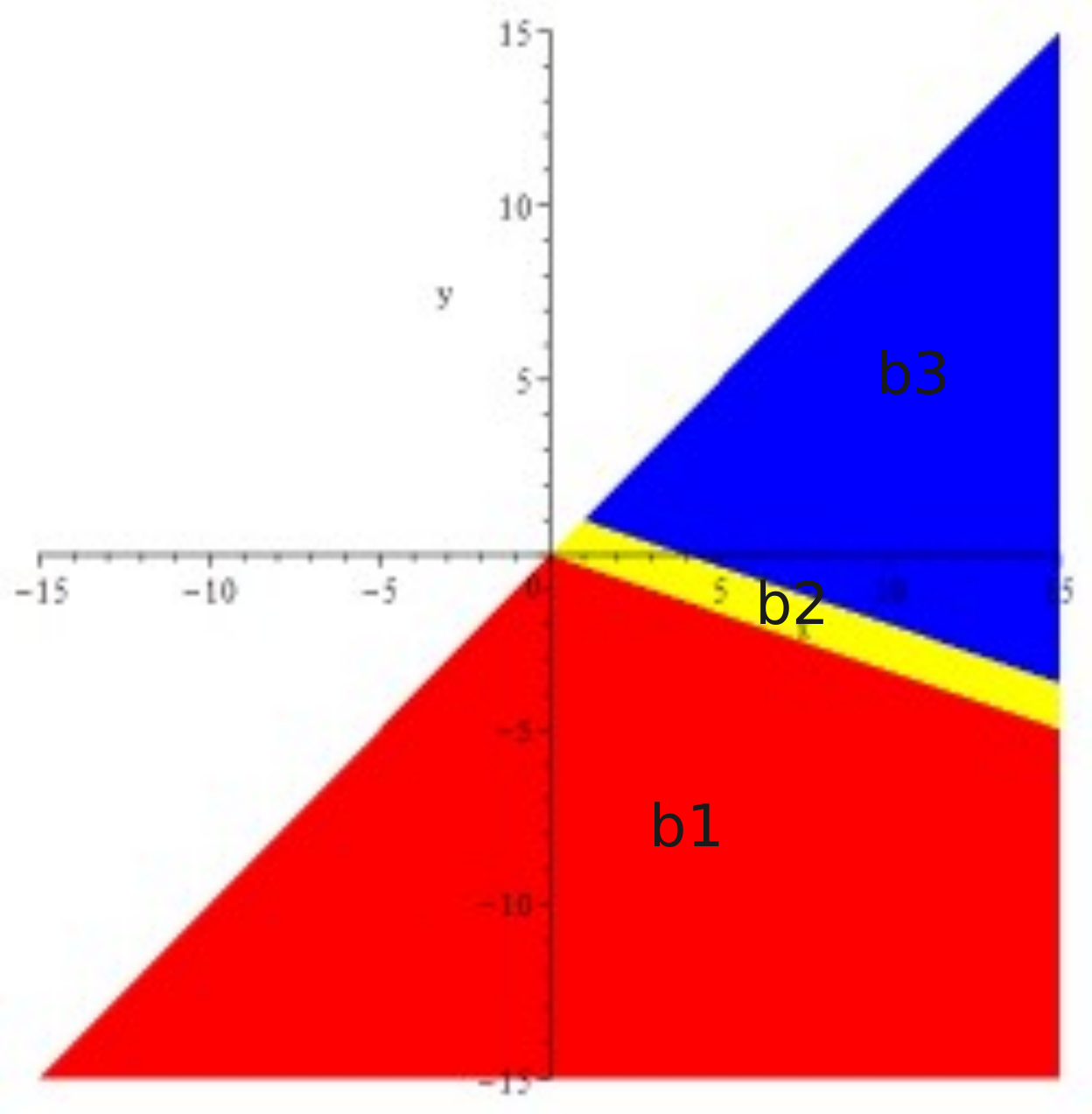}}	
        	\subfigure[ $f_6^{**}$   ]{
        		\centering
        		\includegraphics[height=80px]{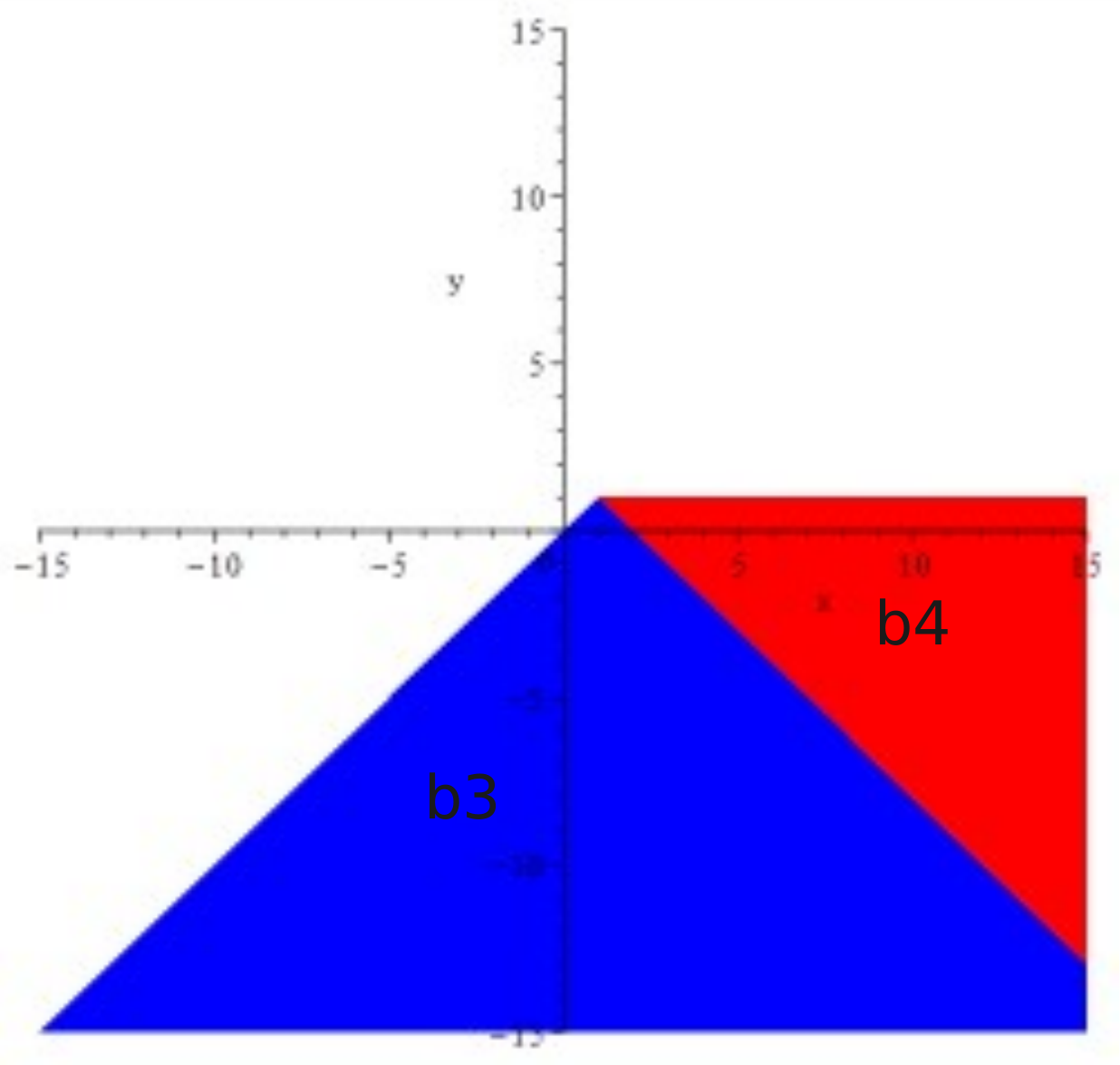}}	
        	\subfigure[ $f_7^{**}$   ]{
        		\centering
        		\includegraphics[height=80px]{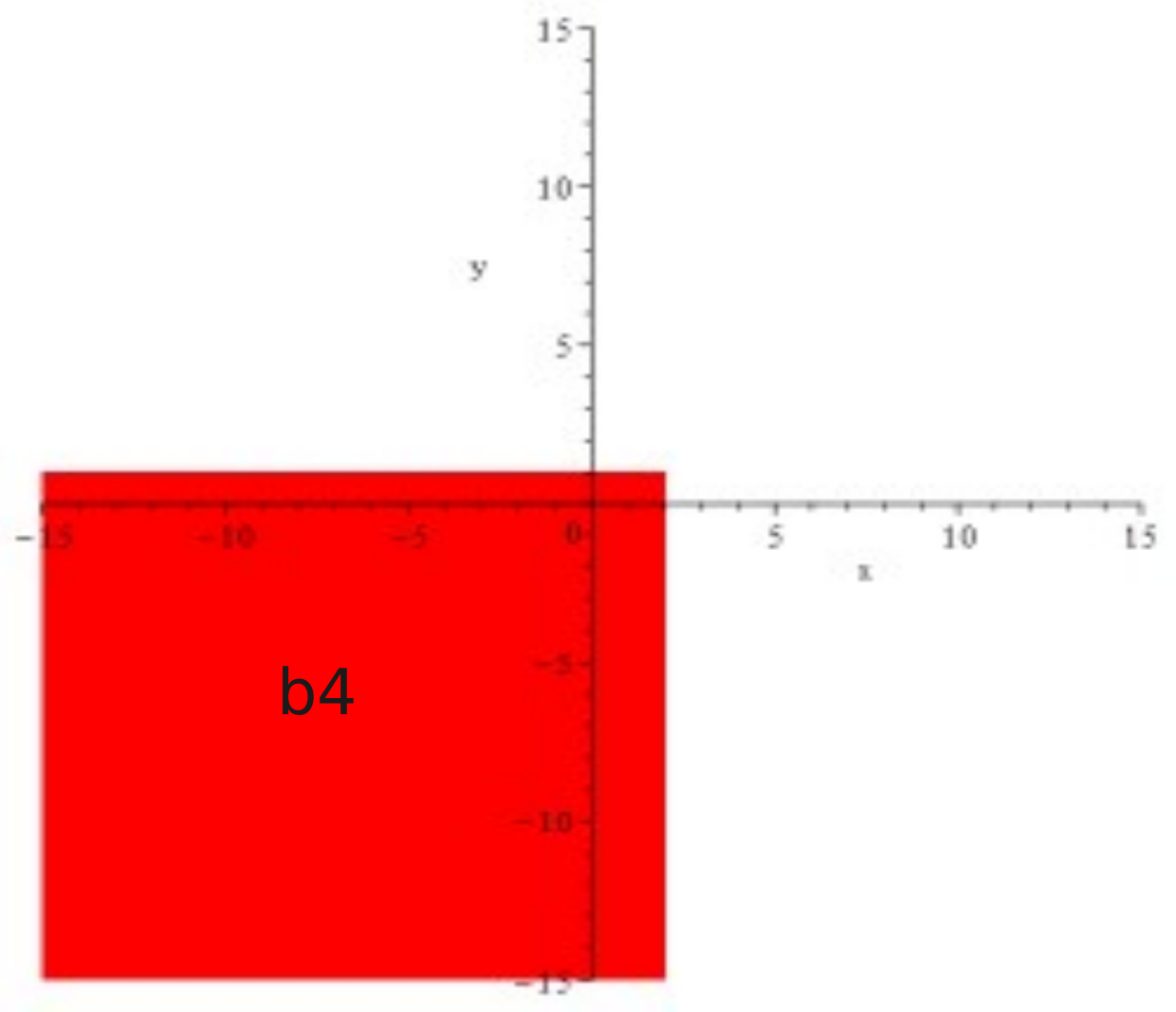}}	
        	
        	\caption{$f_i^{**}$ for Piece i.(Functions are defined in Table~\ref{table:Fig3})}
        	\label{fig:exp}
        \end{figure}

        \begin{table}[tbph]
        \centering
        \caption{Function definitions for Figure \ref{fig:exp}.}
        	\label{table:Fig3}
        \begin{tabular}{@{}llr@{}} \toprule 
            Piece Nunber & Expression\\
            \midrule
        	b1 &$0$ \\
            b2 &$2y^2/(y-x+2)$ \\
            b3 &$x/2+3y/2-1$ \\
            b4 &$x+2y-2$ \\
        \end{tabular}
        \end{table}         
                            
        Taking the maximum over all conjugates we get the biconjugate as shown in Figure~\ref{fig:biconjugate}.

        \begin{figure}
        	\centering
        	\subfigure[$f^{**} = \max_i (f_i^{**})$ ]{
        		\centering
        		\includegraphics[height=80px]{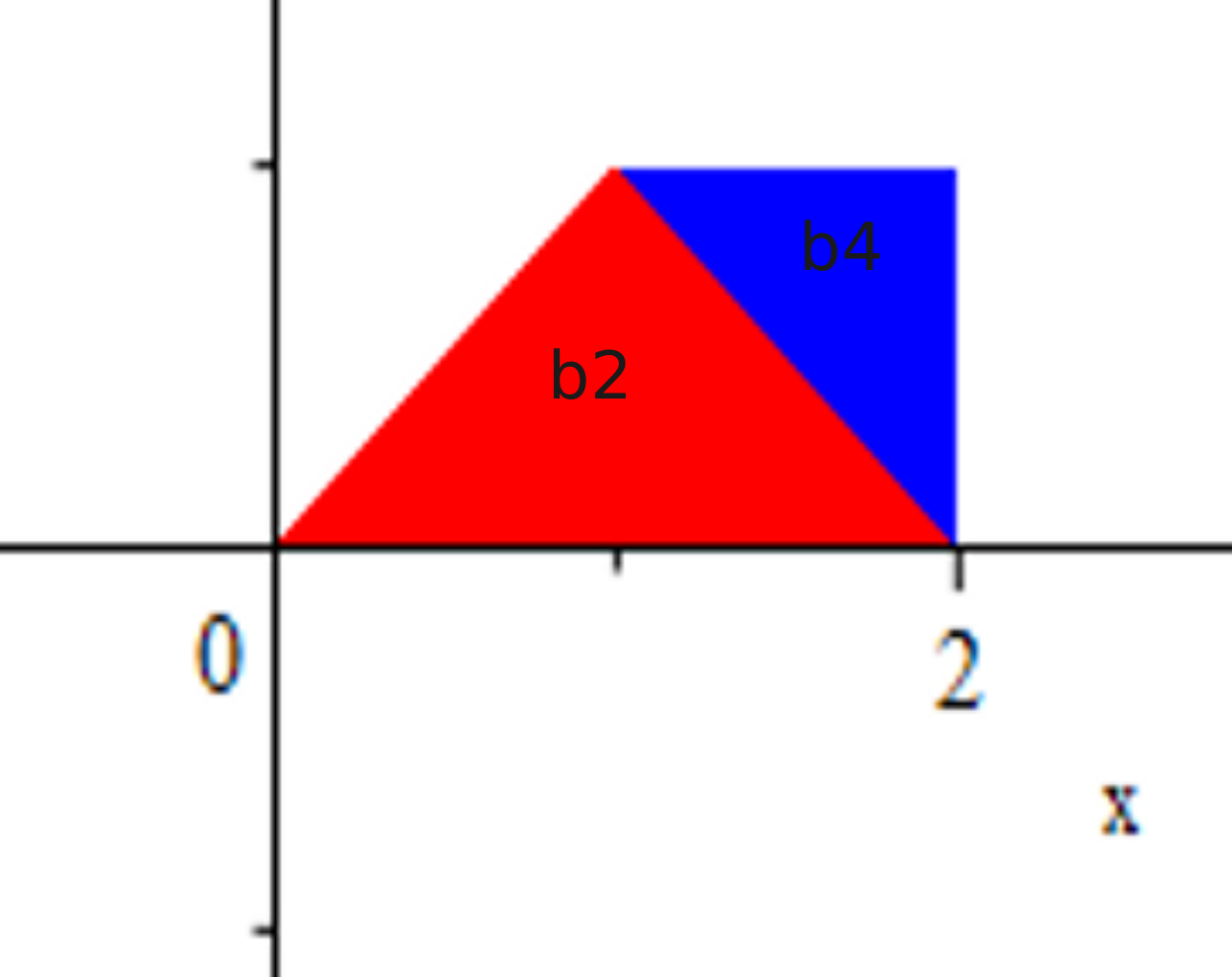}}
        	\caption{Biconjugate(Functions are defined in Table~\ref{table:Fig3})}
        	\label{fig:biconjugate}
        \end{figure}
        \FloatBarrier
\end{example}

\begin{example}\label{ex2}    

        Figure~\ref{fig:exp2}[a] is the conjugate of a PLQ function, the biconjugate is given in Figure~\ref{fig:exp2}[b].

        \begin{figure}%
	\centering
	\subfigure[Conjugate]{
		\centering
        \includegraphics[width=.45\textwidth]{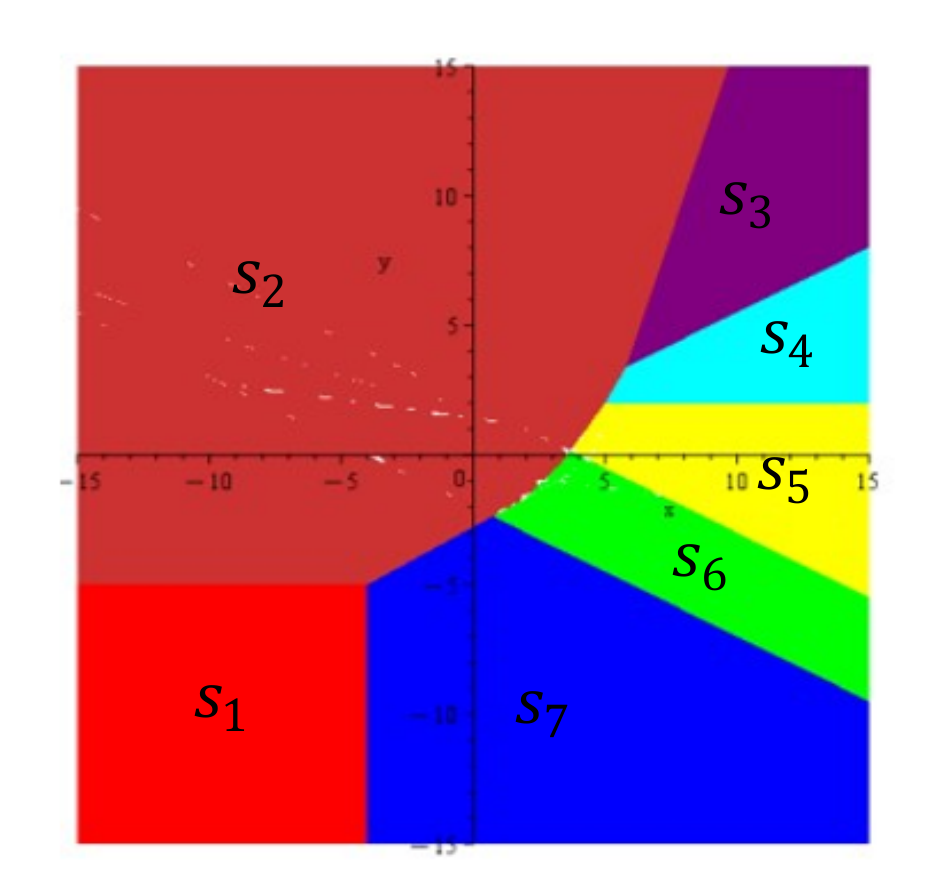}}%
        \hfill
    \subfigure[Biconjugate]{
		\centering		%
        \includegraphics[width=.45\textwidth]{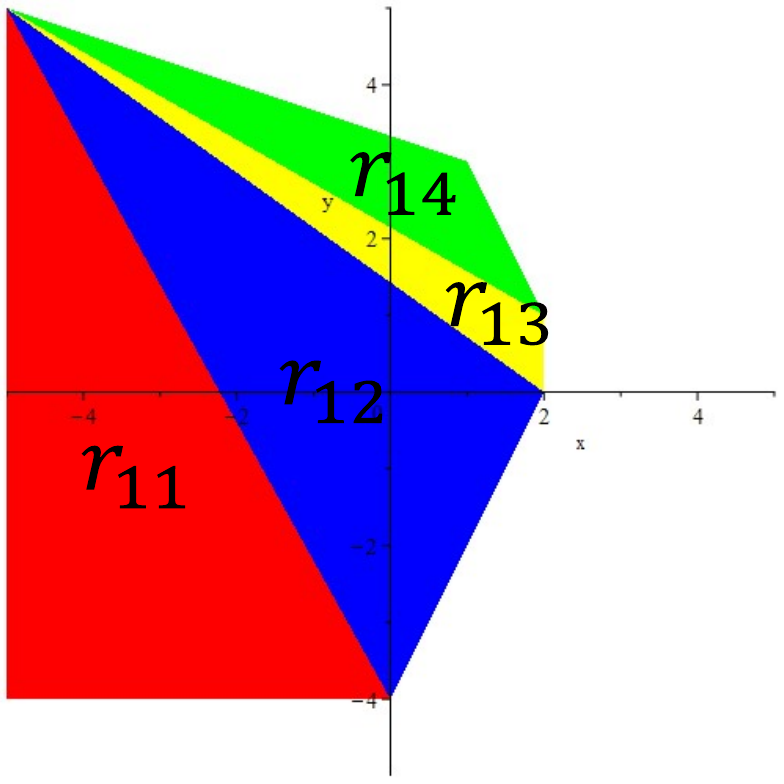}}	\\

	\caption{Conjugate of Example 2.(The functions are enumerated in the Table \ref{table:fig4})}
	\label{fig:exp2}
\end{figure}

\begin{table}[tbph]
\centering
\caption{Function definitions for Figure \ref{fig:exp2}.}
	\label{table:fig4}
\begin{tabular}{@{}llr@{}} \toprule 
    Function name & Expression\\
    \midrule
	$r_{11}$ &$-4x-5y-20$  \\ 
    $r_{12}$ &$(30x - 5y + 4xy+10x^2+5y^2-100)/(2x - y + 15)$  \\ 
    $r_{13}$ &$5x + 2y - 10$  \\ 
    $r_{14}$ &$29x/5 + 17y/5-13$  \\ 
    $s_{1}$ &$-5s_1+4s_2-20$  \\ 
    $s_{2}$ &$-5s_1+5s_2+25$  \\ 
    $s_{3}$ &$s_1+3s_2-3$  \\ 
    $s_{4}$ &$2s_1+s_2-2$  \\ 
    $s_{5}$ &$2s_1$  \\ 
    $s_{6}$ &$s_1^2/8+s_1s_2/2+s_1+s_2^2/8-2s_2+2$  \\ 
    $s_{7}$ &$-4s_2$  \\ 
	\bottomrule 
\end{tabular}
\end{table}

        \FloatBarrier
\end{example}

        \FloatBarrier

\section{Conclusion}\label{s:conclusion}
    Previous work showed a linear-time algorithm to compute the conjugate of a bivariate PLQ function. Here we show a linear time algorithm to compute the biconjugate thereby the conjugate and biconjugate of a bivariate PLQ function can be obtained in linear time. 
    
    While it was known that the biconjugate admits a polyhedral subdivision, its expression was unknown. We show that it is a rational function.
    
    As future work, once we have the subdivision of the domain of the biconjugate, we can directly compute the expressions using the same formulae obtained in \ref{s:convexenvelope_formulae}. 
    
    After observing experimental results, it seems possible to predict the subdivision of the domain of the biconjugate from the domain of the original PLQ function. This seems possible when the original domain is convex and connected. Experimentally, there seem to be patterns emerging in the domain of the biconjugate, hence the next step would be to compute these divisions directly. If we can get direct formulae, then it might be possible to recognize patterns and extend the work to three or more variables. 
    
    Another possible direction is to explore non-convex domains for PLQ functions.
    
    From a coding perspective, the code could be parallelized resulting in faster computation.
        
    From an application perspective, the convex envelope is used to determine lower bounds while looking for a global minimum, for example in a branch-and-bound algorithm. It would be interesting to look into whether having explicit expressions that are quickly computed help such computation.
    
\begin{appendix}
\appendixpage
    \section{Convex envelope of a bilinear function over a triangle}\label{s:convexenvelope_formulae}
        When a function is restricted to an edge of the triangle and the resulting function is strictly convex, we call the edge a convex edge. In case of bilinear functions, when the slope is positive, the edge is convex. The triangular pieces can be classified by the number of convex edges they have. Thus we have the following triangles:
    \begin{enumerate}
        \item No Convex edges.
        \item One Convex edge.
        \item Two Convex edges.
        \item Three Convex edges.
    \end{enumerate}

    The paper~\cite{LOCATELLI-16} gives us a method to compute the convex envelope of a bilinear function on a polytope. We have solved these optimization problems symbolically to obtain direct formulae to obtain the convex envelopes. This gives us the envelopes in one step without having to solve an exponential number of optimization problems. We provide the formulae in this section using the same notations as~\cite{LOCATELLI-16}, and refer to that paper for more details.
    
    \subsection{No Convex edges}\label{noconvexEdge}
        When there are no convex edges, the set of convex edges is empty while the set of vertices has all three vertices.
    
        \subsubsection{Optimization Problem}
    
        Subproblems for every pair \((i,j), i,j \in V_P, i \neq j\) are 
        \begin{equation*}
            \begin{aligned}
                \max \{ \eta_{i}(a,b) +ax + by \;: \;  
                    &\eta_{i}(a,b)=\eta_{j}(a,b), \\
                    & \eta_{i}(a,b)\leq \eta_{k}(a,b), k \neq i, k \neq j,\\
                    & (a,b) \in S_{r} \}.
            \end{aligned}
        \end{equation*}

        In this case, we have three optimization problems for vertex pairs (1,2), (1,3) and (2,3) as follows   
        \begin{equation*} 
            \begin{aligned}
                \max \{ \eta_{i}(a,b) +ax + by \;: \;  
                    &\eta_{i}(a,b)=\eta_{j}(a,b)\}.
            \end{aligned}
        \end{equation*}
        The solutions to all of these is the same as symbolic verified in ``LinearlinearT.m''.
    
        The convex envelope for a bilinear form defined on a triangle with vertices $(x_1,y_1)$,$(x_2,y_2)$,$(x_3,y_3)$ is given by $ax + by + c$
        where
        \begin{align*}
        a &=\frac{(x_1y_1y_2 - x_1y_1y_3 - x_2y_1y_2 + x_2y_2y_3 + x_3y_1y_3 - x_3y_2y_3)}{(x_1y_2 - x_2y_1 - x_1y_3 + x_3y_1 + x_2y_3 - x_3y_2)},\\
        b &=\frac{(x_1x_2y_2 - x_1x_2y_1 + x_1x_3y_1 - x_1x_3y_3 - x_2x_3y_2 + x_2x_3y_3)}{(x_1y_2 - x_2y_1 - x_1y_3 + x_3y_1 + x_2y_3 - x_3y_2)},\\
        c &=\frac{(x_1x_2y_1y_3 - x_1x_3y_1y_2 - x_1x_2y_2y_3 + x_2x_3y_1y_2 + x_1x_3y_2y_3 - x_2x_3y_1y_3)}{(x_1y_2 - x_2y_1 - x_1y_3 + x_3y_1 + x_2y_3 - x_3y_2)}.
        \end{align*}

    \begin{example}\label{Example1}
        The Convex Envelope of $f(x,y)=xy$ over $\co \{(1,-1), (-1,-1). (-1,1)\}$ is $(x,y)\mapsto -x-y-1$ on the same domain; see Figure~\ref{fig:ce1}.
        
        \begin{figure}
            \centering
            \includegraphics[width=.5\textwidth,trim=170 180 160 130, clip]{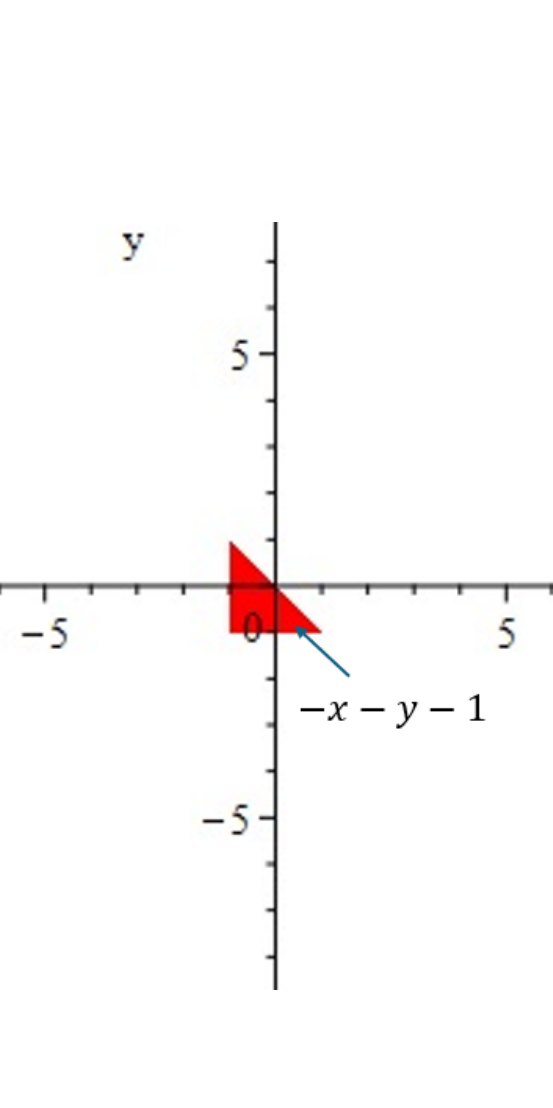}
            \caption{Convex Envelope for Example \ref{Example1}.}
            \centering
            \label{fig:ce1}
        \end{figure}
    \end{example}

    \subsection{One Convex edge}
        Here we have one convex edge $y=mx+q$, hence the set $E_P$ has one edge and the set $V_P$ has one vertex. Using this information and referring to \cite{LOCATELLI-16} we solve the following optimization problems in order to obtain the convex envelope.
    
        Subproblems for every pair \((i,j), i,j \in E_P \cup V_P, i \neq j\) are given as 
        \begin{equation*} 
            \begin{aligned}
                \max \{ \eta_{i}(a,b) +ax + by \;: \;  
                    &\eta_{i}(a,b)=\eta_{j}(a,b), \\
                    & \eta_{i}(a,b)\leq \eta_{k}(a,b), k \neq i, k \neq j,\\
                    & (a,b) \in S_{r} \}.
            \end{aligned}
        \end{equation*}

        The $\eta$ functions obtained are:
        \begin{itemize}
            \item For $V$, $\eta_w = f(x_1,y_1) - ax_1 - by_1 = x_1y_1 - ax_1 - by_1.$
            \item For $E$,  $\eta_h = -(a+mb-q)^2)/(4m)-bq$ for edge $y=mx+q.$
            \item For $E^-$, $\eta_h^- = f(x_2,y_2) - ax_2 - by_2 = x_2y_2 - ax_2 - by_2$.
            \item For $E^+$, $\eta_h^+ = f(x_3,y_3) - ax_3 - by_3 = x_3y_3 - ax_3 - by_3$.
        \end{itemize}
        Thus we have three optimization problems corresponding to $V - E^-$, $V - E$ and $V - E^+$.

        \subsubsection{Optimization Problem (V - E)}

        In this case, $\eta_h = -\frac{(a+mb-q)^2}{4m}-bq$ corresponding to the edge, while $\eta_w = x_1y_1 - ax_1 - by_1$ corresponding to the vertex $(x_1,y_1)$. Thus the optimization problem can be written as
        \begin{equation} \label{optvE}
            \max \{ \eta_w +ax + by \; : \;  \eta_w=\eta_h, (a,b) \in S_{r} \}.
        \end{equation}

        By solving this problem symbolically we obtain the objective as a rational function, given by
        \begin{equation} \label{eq:rational1}
        r(x,y)=\frac{ax^2+bxy+cy^2+dx+ey+f}{gx+hy+k},
        \end{equation}
        where $a=-my_1$, $b=q$, $c=x_1$, $d=-qy_1+mx_1y_1$, $e=-qx_1-x_1y_1$, $f=qx_1y_1$, $g=-m$, $h=1$ and $k=-y_1+ mx_1$; or we get
        \[r(x,y)=\frac{\xi_1(x,y)^2}{\xi_2(x,y)}+\xi_0(x,y)\]
        where $\xi_1(x,y) = ax+by+c$ with
        \begin{align*}        
            a=\frac{q-2y_1}{4(q-y_1+mx_1)},\\
            b=\frac{q+2mx_1}{4m(q-y_1+mx_1)},\\
            c=\frac{q(-y_1-mx_1)}{4m(q-y_1+mx_1)};
        \end{align*}
        $\xi_2(x,y) = ax+by+c$ with
        \begin{align*}
        a&=\frac{-1}{4(q-y_1+mx_1)},\\
        b&=\frac{1}{4m(q-y_1+mx_1)},\\
        c&=\frac{(-y_1+mx_1)}{4m(q-y_1+mx_1)};
        \end{align*}
        and $\xi_0(x,y) = ax+by+c$ with 
        \begin{align*}
            a&=\frac{q^2+4mx_1y_1}{4(q-y_1+mx_1)},\\
            b&=\frac{-q^2-4mx_1y_1}{4m(q-y_1+mx_1)},\\
            c&=\frac{(q^2y_1-mq^2x_1+4mqx_1y_1)}{4m(q-y_1+mx_1)}.
        \end{align*}

        The edges $du\xi_2=\xi_1$ and $dl\xi_2=\xi_1$ are the two sides of original triangle, hence we do not need to compute the other two linear functions as given in \cite{LOCATELLI-16} as the solution to the optimization problem (\ref{optvE}). 
        
        Rational function evaluated at the vertex $(x_1,y_1)$ is $r(x_1,y_1) = 0/0$ where $\lim_{(x,y) \mapsto (x_1,y_1)} = f(x_1,y_1) $ while $r(x_2,y_2) = f(x_2,y_2)$ and $r(x_3,y_3) = f(x_3,y_3)$.
        
        Verified symbolically in vertexNan.m

        \subsubsection{Optimization Problems \texorpdfstring{$(V - E^-)$ and $(V - E^+)$}{Optimization Problems (V - E- and V - E+)}}

        The solutions obtained by solving  the other two optimization problems, ($V - E^-$) and $V - E^+$ are either infeasible or lie below the solution obtained by solving $V - E$.         
        Thus the rational function obtained in \eqref{eq:rational1} is the convex envelope over the entire triangle.

        \subsubsection{Example 2} \label{Example2}
         The Convex Envelope of $f(x,y)=x y$ over $\co \{(1,1), (0,0), (2,0)\}$ is the function $(x,y)\mapsto (2y^2)/(y - x + 2)$ on the same domain as shown in Figure~\ref{fig:ce1a}.
        
        \begin{figure}
            \centering
            \includegraphics[width=.45\textwidth,trim=150 150 100 100, clip]{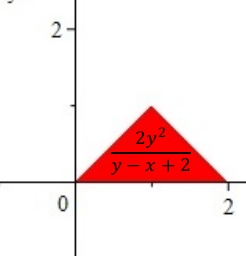}
            \caption{Convex Envelope.}
            \centering
            \label{fig:ce1a}
        \end{figure}

        \subsection{Two Convex edges}
        Here we have two convex edge $y=mx+q$, hence the set $E_P$ has two edges and the set $V_P$ is empty. Using this information and referring to \cite{LOCATELLI-16} we solve the following optimization problems in order to obtain the convex envelope.

        \begin{equation*}
            \begin{aligned}
                \max \{ \eta_{i}(a,b) +ax + by \;: \;  
                    &\eta_{i}(a,b)=\eta_{j}(a,b), \\
                    & \eta_{i}(a,b)\leq \eta_{k}(a,b), k \neq i, k \neq j,\\
                    & (a,b) \in S_{r} \},
            \end{aligned}
        \end{equation*}
        where
        $V_1(x_1,x_2)$, $V_2(x_2,y_2)$ and $V_3(x_3,y_3)$ are the vertices of the triangle with $V_1(x_1,x_2)$ the common vertex between the two convex edges. Thus
        \[
          x_1 = -\frac{q_h-q_w}{m_h-m_w}, y_1 = \frac{m_hq_w-m_wq_h}{m_h-m_w},
          x_2=vx_1, y_2=q_h+m_hvx_1, 
        \]
        and $x_3=v x_2$, $y_3=q_w+m_w v x_2$.

        \subsubsection{Eta functions.}
            Corresponding to the two convex edges, ($V_2-V_1$) and ($V_3-V_1$) we get the following $\eta$ functions.
            
            \begin{itemize}
                \item For $E_1$,  $\eta_h = -(a+m_hb-q_h)^2)/(4m_h)-bq_h$ 
                \item For $E_1^-$, $\eta_h^- = f(x_2,y_2) - ax_2 - by_2 = x_2y_2 - ax_2 - by_2$.
                \item For $E_1^+$, $\eta_h^+ = f(x_1,y_1) - ax_1 - by_1 = x_1y_1 - ax_1 - by_1$.
                \item For $E_2$,  $\eta_w = -(a+m_wb-q_w)^2)/(4m_w)-bq_w$ 
                \item For $E_2^-$, $\eta_w^- = f(x_3,y_3) - ax_3 - by_3 = x_3y_3 - ax_3 - by_3$.
                \item For $E_2^+$, $\eta_w^+ = f(x_1,y_1) - ax_1 - by_1 = x_1y_1 - ax_1 - by_1$.
            \end{itemize}        
        Thus we get nine optimization problems by taking all the pairs. Of these problems only one, $E_1 - E_2$ is feasible.      
            
        \subsubsection{Optimization Problem  (\texorpdfstring{$E_1 - E_2$}))}
            In this problem, we consider the pair of $\eta$ functions corresponding to the two convex edges. Thus 
            \[\eta_h = -\frac{(a+m_hb-q)^2}{4m_h}-bq_h \text{ and } \eta_w = -\frac{(a+m_wb-q)^2}{4m_w}-bq_w.\] 
            The optimization problem can now be written as             
            \[\max \{ \eta_h +ax + by : \eta_w=\eta_h,  (a,b) \in S_{r} \}.\]
            Reformulating the constraint $\eta_h=\eta_w$, we get
            \begin{gather*}                        
               -\frac{(a+m_hb-q)^2}{4m_h}-bq_h = -\frac{(a+m_wb-q)^2}{4m_w}-bq_w, \\
               (\frac{m_h - m_w}{4m_hm_w})a^2 -2 (\frac{m_hq_w - m_wq_h}{4m_hm_w})a + c = 0,
            \end{gather*}
            \[\text{where }c = \frac{- b^2m_h^2m_w + b^2m_hm_w^2 - 2bm_hm_wq_h + 2bm_hm_wq_w + m_hq_w^2 - m_wqh^2}{4m_hm_w}.\]            
            Solving for a \label{eq_alpha}, we get 
            
            \[a = \frac{m_hq_w - m_wq_h \pm \sqrt{m_hm_w} (q_h- q_w+ bm_h - bm_w)}{m_h - m_w}.\]
            
            Let $a = \alpha_1b+\alpha_0$ with 
            \[\alpha_1=\sqrt{m_hm_w},\]
            \[\alpha_0=\frac{m_hq_w - m_wq_h + \sqrt{m_hm_w} (q_h- q_w)}{m_h - m_w}.\]

            The objective function is $\eta_h + ax + by$, by substituting $\eta_h$ and $a$ we get objective in terms of $b$.
            Similarly we get the bounds for $b$ by substituting for a in $l_b \le a+mb \le u_b$. 
            Thus we get the convex expression as 
            \[ax^2 + bxy + cy^2 + dx + ey +f \]
            where
            \begin{align*}
                a &= \frac{m_hm_w}{m_h + m_w + 2\sqrt{m_hm_w}},\\
                b &= \frac{2\sqrt{m_hm_w}}{m_h + m_w + 2\sqrt{m_hm_w}},\\
                c &= \frac{1}{m_h + m_w + 2\sqrt{m_hm_w}},\\
                d &= \frac{(m_hq_w + m_wq_h)}{m_h + m_w + 2\sqrt{m_hm_w}},\\
                e &= \frac{(- (q_h + q_w))}{m_h + m_w + 2\sqrt{m_hm_w}},\\
                f &= \frac{q_hq_w}{m_h + m_w + 2\sqrt{m_hm_w}}.
            \end{align*}
            or
            \[a = \frac{m_hm_w}{m_h + m_w - 2\sqrt{m_hm_w}}\]
            \[b = \frac{-2\sqrt{m_hm_w}}{m_h + m_w - 2\sqrt{m_hm_w}}\]
            \[c = \frac{1}{m_h + m_w - 2\sqrt{m_hm_w}}\]
            \[d = \frac{(m_hq_w + m_wq_h)}{m_h + m_w - 2\sqrt{m_hm_w}}\]
            \[e = \frac{(- (q_h + q_w))}{m_h + m_w - 2\sqrt{m_hm_w}}\]
            \[f = \frac{q_hq_w}{m_h + m_w - 2\sqrt{m_hm_w}}\]            
            By checking the bounds we get the domain as the entire triangle.

            \subsubsection{Example}
            Consider $f(x,y)=xy$ defined on the triangle with vertices, $V_1=(2,1), V_2=(0,0), V_3=(1,0)$. The vertex $V_1$ is the common vertex between the two convex edges. For the first convex edge $V_1-V_3$, $m_h = 1$ and $q_h = -1$. For the second convex edge, $V_1-V_2$, $m_w = \frac{1}{2}$ and $q_w = 0$.
                
            The convex envelope has the same domain as $f$ with expression
            \[(x,y)\mapsto \frac{x^2 + 2\sqrt{2}xy + 2y^2  - x  + 2y  }{3 + 2\sqrt{2}}.\]            
            It is obtained from the second solution for $a$ in Section \ref{eq_alpha} by selecting the solution that is lower than $f$.  

        \subsection{Three Convex edges}
            In case we get a triangle with three convex edges, we can split it into two triangles with two convex edges each by adding a horizontal edge from the middle vertex. 

    \section{Conjugate of convex envelope of a bilinear function over a triangle}\label{s:conjugate_formulae}
       We improve the computational efficiency of~\cite{KUMAR-19} by symbolically computing formulae to obtain the conjugate corresponding to the convex envelopes in the above section.
        
        \subsection{Triangle with no Convex edges}
        The conjugates corresponding to the interior and edges are contained in rays. Corresponding to the three vertices, we get unbounded polyhedral regions with $s_1v_x+s_2v_y-f(v_x,v_y)$ defined for each vertex $(v_x,v_y)$.

        \subsection{Triangle with one convex edge}
            The conjugate corresponding to the interior is contained in rays \cite{KUMAR-19}.

            \subsubsection{Domain of the conjugate}
            Name Vertex 1 a vertex where $r(x,y)=0/0$. Its subdifferential has a parabolic inequality, while the subdifferential at edge (2,3) is the other side of the same parabolic inequality. For the other two edges, we get polyhedral subdivisions; see \cite{KUMAR-19}.
    
            The equation of the parabola is computed as $as_1^2+bs_1s_2+cs_2^2+ds_1+es_2+f$
            where $a=-1$, $b=-2m$, $d=2q+4mx_1$, $c=-m^2$, $e=-(2mq - 4my_1)$, and $f=-(q^2 + 4mx_1y_1)$.

            \subsubsection{Expression at vertices}
                At Vertex $(v_x,v_y)$, the conjugate expression is 
                \begin{equation}\label{eq:vertex}
                (s_1,s_2)\mapsto s_1v_x+s_2v_y-f(v_x,v_y).    
                \end{equation}                
                
            \subsubsection{Expression at edge} \label{conjExpr}
                The conjugate expression on an edge $y=mx+q$ is given by $ax^2+bxy+cy^2+dx+ey+f$,
                where
                 \[a = \frac{1}{4m},
                 b = \frac{1}{2},
                 c = \frac{m}{4},
                 d = -\frac{q}{2m},
                 e =  \frac{q}{2},
                 f = \frac{q^2}{4m}.\]                
                At the other two edges, we get the linear function \eqref{eq:vertex}.
                
        \subsection{Triangle with two convex edges}
            The conjugates corresponding to the interior is contained in a ray. Corresponding to the three vertices, we get unbounded polyhedral regions with $s_1v_x+s_2v_y-f(v_x,v_y)$ defined for each vertex $(v_x,v_y)$. Corresponding to the edges we get three linear edges, two of which are parallel. For the two convex edges we obtain the expression 
            \[\frac{m_h^2y^2 + 2m_hq_hy + 2m_hxy + q_h^2 - 2q_hx + x^2}{4m_h},
            \]
            when the edge is $y = m_h x+q_h.$ which is same as the expression in Section~\ref{conjExpr}.
        
    \section{Conjugate Expression: Code}\label{s:conjugate_expr}
   
    Conjugate expression, $f^*(s) = \sup(s.x-f(x))$ for $f$ restricted to an edge.
    \begin{lstlisting}[escapeinside={(*}{*)}]
(*function conj = conjugateExpr(edge,f,x,y)*)
    lambda = sym('lambda');
    infs = sym('inf');
    s1 = sym('s1');
    s2 = sym('s2');

    %edge
    (*dedge = sym.empty();*)
    (*dedge(1) = diff(edge,x);*)
    (*dedge(2) = diff(edge,y);*)

    %f
    (*df = sym.empty();*)
    (*df(1) = diff(f,x);*)
    (*df(2) = diff(f,y);*)
    
    (*eq1 = s1 -  df(1) - lambda*dedge(1);*)
    (*eq2 = s2 -  df(2) - lambda*dedge(2);*)
    
    (*xyl = solve([eq1,eq2,edge],[x,y,lambda]);*)
    
     (*conj = infs;*)
    
    (*if isempty(xyl)*)
    (*   return*)
    (*elseif isempty(xyl.x) | isempty(xyl.y) | isempty(xyl.lambda)*)
    (*   return*)
    (*end*)
    
    (*conj = s1*xyl.x + s2*xyl.y - subs(f,[x,y],[xyl.x,xyl.y]);*)
    (*conj = simplifyFraction(conj););*)
    (*conj = subs(conj,[s1,s2],[x,y]););*)

(*end*)
\end{lstlisting}

\end{appendix}

\bibliography{biconjugate}

\def\cprime{$'$}

\begin{thebibliography}{45}
\ifx \bisbn   \undefined \def \bisbn  #1{ISBN #1}\fi
\ifx \binits  \undefined \def \binits#1{#1}\fi
\ifx \bauthor  \undefined \def \bauthor#1{#1}\fi
\ifx \batitle  \undefined \def \batitle#1{#1}\fi
\ifx \bjtitle  \undefined \def \bjtitle#1{#1}\fi
\ifx \bvolume  \undefined \def \bvolume#1{\textbf{#1}}\fi
\ifx \byear  \undefined \def \byear#1{#1}\fi
\ifx \bissue  \undefined \def \bissue#1{#1}\fi
\ifx \bfpage  \undefined \def \bfpage#1{#1}\fi
\ifx \blpage  \undefined \def \blpage #1{#1}\fi
\ifx \burl  \undefined \def \burl#1{\textsf{#1}}\fi
\ifx \doiurl  \undefined \def \doiurl#1{\url{https://doi.org/#1}}\fi
\ifx \betal  \undefined \def \betal{\textit{et al.}}\fi
\ifx \binstitute  \undefined \def \binstitute#1{#1}\fi
\ifx \binstitutionaled  \undefined \def \binstitutionaled#1{#1}\fi
\ifx \bctitle  \undefined \def \bctitle#1{#1}\fi
\ifx \beditor  \undefined \def \beditor#1{#1}\fi
\ifx \bpublisher  \undefined \def \bpublisher#1{#1}\fi
\ifx \bbtitle  \undefined \def \bbtitle#1{#1}\fi
\ifx \bedition  \undefined \def \bedition#1{#1}\fi
\ifx \bseriesno  \undefined \def \bseriesno#1{#1}\fi
\ifx \blocation  \undefined \def \blocation#1{#1}\fi
\ifx \bsertitle  \undefined \def \bsertitle#1{#1}\fi
\ifx \bsnm \undefined \def \bsnm#1{#1}\fi
\ifx \bsuffix \undefined \def \bsuffix#1{#1}\fi
\ifx \bparticle \undefined \def \bparticle#1{#1}\fi
\ifx \barticle \undefined \def \barticle#1{#1}\fi
\bibcommenthead
\ifx \bconfdate \undefined \def \bconfdate #1{#1}\fi
\ifx \botherref \undefined \def \botherref #1{#1}\fi
\ifx \url \undefined \def \url#1{\textsf{#1}}\fi
\ifx \bchapter \undefined \def \bchapter#1{#1}\fi
\ifx \bbook \undefined \def \bbook#1{#1}\fi
\ifx \bcomment \undefined \def \bcomment#1{#1}\fi
\ifx \oauthor \undefined \def \oauthor#1{#1}\fi
\ifx \citeauthoryear \undefined \def \citeauthoryear#1{#1}\fi
\ifx \endbibitem  \undefined \def \endbibitem {}\fi
\ifx \bconflocation  \undefined \def \bconflocation#1{#1}\fi
\ifx \arxivurl  \undefined \def \arxivurl#1{\textsf{#1}}\fi
\csname PreBibitemsHook\endcsname

\bibitem[\protect\citeauthoryear{Hiriart-Urruty
  et~al.}{2011}]{HIRIART-URRUTY-11}
\begin{barticle}
\bauthor{\bsnm{Hiriart-Urruty}, \binits{J.-B.}},
\bauthor{\bsnm{L{\'o}pez}, \binits{M.A.}},
\bauthor{\bsnm{Volle}, \binits{M.}}:
\batitle{The {$\epsilon$}-strategy in variational analysis: illustration with
  the closed convexification of a function}.
\bjtitle{Revista Matem\'atica Iberoamericana}
\bvolume{27}(\bissue{2}),
\bfpage{449}--\blpage{474}
(\byear{2011})
\doiurl{10.4171/RMI/643}
\end{barticle}
\endbibitem

\bibitem[\protect\citeauthoryear{Locatelli and Schoen}{2013}]{LOCATELLI-13}
\begin{bbook}
\bauthor{\bsnm{Locatelli}, \binits{M.}},
\bauthor{\bsnm{Schoen}, \binits{F.}}:
\bbtitle{Global Optimization}.
\bsertitle{{MOS}-{SIAM} Series on Optimization},
p. \bfpage{432}.
\bpublisher{Society for Industrial and Applied Mathematics},
\blocation{Philadelphia, PA}
(\byear{2013}).
\burl{http://epubs.siam.org/doi/abs/10.1137/1.9781611972672.fm}
\end{bbook}
\endbibitem

\bibitem[\protect\citeauthoryear{Crama}{1989}]{CRAMA-89}
\begin{barticle}
\bauthor{\bsnm{Crama}, \binits{Y.}}:
\batitle{Recognition problems for special classes of polynomials in 0–1
  variables}.
\bjtitle{Mathematical Programming}
\bvolume{44},
\bfpage{139}--\blpage{155}
(\byear{1989})
\end{barticle}
\endbibitem

\bibitem[\protect\citeauthoryear{Locatelli and Schoen}{2014}]{LOCATELLI-14}
\begin{barticle}
\bauthor{\bsnm{Locatelli}, \binits{M.}},
\bauthor{\bsnm{Schoen}, \binits{F.}}:
\batitle{On convex envelopes for bivariate functions over polytopes}.
\bjtitle{Mathematical Programming}
\bvolume{144}(\bissue{1-2, Ser. A}),
\bfpage{65}--\blpage{91}
(\byear{2014})
\doiurl{10.1007/s10107-012-0616-x}
\end{barticle}
\endbibitem

\bibitem[\protect\citeauthoryear{Locatelli}{2014}]{LOCATELLI-14a}
\begin{barticle}
\bauthor{\bsnm{Locatelli}, \binits{M.}}:
\batitle{A technique to derive the analytical form of convex envelopes for some
  bivariate functions}.
\bjtitle{Journal of Global Optimization}
\bvolume{59}(\bissue{2}),
\bfpage{477}--\blpage{501}
(\byear{2014})
\doiurl{10.1007/s10898-014-0177-z}
\end{barticle}
\endbibitem

\bibitem[\protect\citeauthoryear{Locatelli}{2016}]{LOCATELLI-16}
\begin{barticle}
\bauthor{\bsnm{Locatelli}, \binits{M.}}:
\batitle{Polyhedral subdivisions and functional forms for the convex envelopes
  of bilinear, fractional and other bivariate functions over general
  polytopes}.
\bjtitle{Journal of Global Optimization}
\bvolume{66}(\bissue{4}),
\bfpage{629}--\blpage{668}
(\byear{2016})
\doiurl{10.1007/s10898-016-0418-4}
\end{barticle}
\endbibitem

\bibitem[\protect\citeauthoryear{Locatelli}{2018}]{LOCATELLI-18a}
\begin{barticle}
\bauthor{\bsnm{Locatelli}, \binits{M.}}:
\batitle{Convex envelopes of bivariate functions through the solution of {KKT}
  systems}.
\bjtitle{Journal of Global Optimization}
\bvolume{72}(\bissue{2}),
\bfpage{277}--\blpage{303}
(\byear{2018})
\doiurl{10.1007/s10898-018-0626-1}
\end{barticle}
\endbibitem

\bibitem[\protect\citeauthoryear{Khademnia and Davarnia}{2024}]{KHADEMNIA-24}
\begin{barticle}
\bauthor{\bsnm{Khademnia}, \binits{E.}},
\bauthor{\bsnm{Davarnia}, \binits{D.}}:
\batitle{Convexification of bilinear terms over network polytopes}.
\bjtitle{Mathematics of operations research}
(\byear{2024})
\doiurl{10.1287/moor.2023.0001}
\end{barticle}
\endbibitem

\bibitem[\protect\citeauthoryear{Al-Khayyal and Falk}{1983}]{AL-KHAYYAL-83}
\begin{botherref}
\oauthor{\bsnm{Al-Khayyal}, \binits{F.}},
\oauthor{\bsnm{Falk}, \binits{J.}}:
Jointly constrained biconvex programming.
Mathematics of Operations Research
\textbf{8}
(1983)
\end{botherref}
\endbibitem

\bibitem[\protect\citeauthoryear{Sherali and Alameddine}{1990}]{SHERALI-90}
\begin{barticle}
\bauthor{\bsnm{Sherali}, \binits{H.}},
\bauthor{\bsnm{Alameddine}, \binits{A.}}:
\batitle{An explicit characterization of the convex envelope of a bivariate
  bilinear function over special polytopes}.
\bjtitle{Annals of Operations Research}
\bvolume{25},
\bfpage{197}--\blpage{209}
(\byear{1990})
\end{barticle}
\endbibitem

\bibitem[\protect\citeauthoryear{Linderoth}{2005}]{LINDEROTH-05}
\begin{barticle}
\bauthor{\bsnm{Linderoth}, \binits{J.}}:
\batitle{A simplicial branch-and-bound algorithm for solving quadratically
  constrained quadratic programs}.
\bjtitle{Mathematical Programming}
\bvolume{2},
\bfpage{251}--\blpage{282}
(\byear{2005})
\end{barticle}
\endbibitem

\bibitem[\protect\citeauthoryear{Anstreicher and Burer}{2010}]{ANSTREICHER-10}
\begin{barticle}
\bauthor{\bsnm{Anstreicher}, \binits{K.M.}},
\bauthor{\bsnm{Burer}, \binits{S.}}:
\batitle{Computable representations for convex hulls of low-dimensional
  quadratic forms}.
\bjtitle{Mathematical Programming}
\bvolume{124}(\bissue{1-2}),
\bfpage{33}--\blpage{43}
(\byear{2010})
\doiurl{10.1007/s10107-010-0355-9}
\end{barticle}
\endbibitem

\bibitem[\protect\citeauthoryear{Anstreicher}{2012}]{ANSTREICHER-12}
\begin{barticle}
\bauthor{\bsnm{Anstreicher}, \binits{K.}}:
\batitle{On convex relaxations for quadratically constrained quadratic
  programming}.
\bjtitle{Mathematical Programming}
\bvolume{136},
\bfpage{233}--\blpage{251}
(\byear{2012})
\end{barticle}
\endbibitem

\bibitem[\protect\citeauthoryear{Locatelli}{2024}]{LOCATELLI-24}
\begin{barticle}
\bauthor{\bsnm{Locatelli}, \binits{M.}}:
\batitle{A new technique to derive tight convex underestimators (sometimes
  envelopes)}.
\bjtitle{Computational optimization and applications}
\bvolume{87}(\bissue{2}),
\bfpage{475}--\blpage{499}
(\byear{2024})
\doiurl{10.1007/s10589-023-00534-8}
\end{barticle}
\endbibitem

\bibitem[\protect\citeauthoryear{Rockafellar and Wets}{1998}]{ROCKAFELLAR-98a}
\begin{bbook}
\bauthor{\bsnm{Rockafellar}, \binits{R.T.}},
\bauthor{\bsnm{Wets}, \binits{R.J.-B.}}:
\bbtitle{Variational Analysis}.
\bpublisher{Springer},
\blocation{Springer Dordrecht Heidelberg London New York}
(\byear{1998}).
\burl{http://www.springer.com/math/book/978-3-540-62772-2}
\end{bbook}
\endbibitem

\bibitem[\protect\citeauthoryear{Goebel}{2000}]{GOEBEL-00}
\begin{botherref}
\oauthor{\bsnm{Goebel}, \binits{R.}}:
Convexity, convergence and feedback in optimal control.
PhD thesis,
University of Washington, Seattle
(2000).
\url{https://digital.lib.washington.edu/dspace/handle/1773/5792}
\end{botherref}
\endbibitem

\bibitem[\protect\citeauthoryear{Gardiner et~al.}{2014}]{GARDINER-13}
\begin{barticle}
\bauthor{\bsnm{Gardiner}, \binits{B.}},
\bauthor{\bsnm{Jakee}, \binits{K.}},
\bauthor{\bsnm{Lucet}, \binits{Y.}}:
\batitle{Computing the partial conjugate of convex piecewise linear-quadratic
  bivariate functions}.
\bjtitle{Computational Optimization and Applications}
\bvolume{58}(\bissue{1}),
\bfpage{249}--\blpage{272}
(\byear{2014})
\doiurl{10.1007/s10589-013-9622-z}
\end{barticle}
\endbibitem

\bibitem[\protect\citeauthoryear{Moreau}{1965}]{MOREAU-65}
\begin{barticle}
\bauthor{\bsnm{Moreau}, \binits{J.-J.}}:
\batitle{Proximit\'e et dualit\'e dans un espace {H}ilbertien}.
\bjtitle{Bulletin de la Soci\'et\'e Math\'ematique de France}
\bvolume{93},
\bfpage{273}--\blpage{299}
(\byear{1965})
\end{barticle}
\endbibitem

\bibitem[\protect\citeauthoryear{Brenier}{1989}]{BRENIER-89a}
\begin{barticle}
\bauthor{\bsnm{Brenier}, \binits{Y.}}:
\batitle{Un algorithme rapide pour le calcul de transform\'ees de
  {L}egendre--{F}enchel discr\`etes}.
\bjtitle{Comptes rendus de l'Académie des sciences. Série 1, Mathématique}
\bvolume{308},
\bfpage{587}--\blpage{589}
(\byear{1989})
\end{barticle}
\endbibitem

\bibitem[\protect\citeauthoryear{Corrias}{1996}]{CORRIAS-93a}
\begin{barticle}
\bauthor{\bsnm{Corrias}, \binits{L.}}:
\batitle{Fast {L}egendre--{F}enchel transform and applications to
  {H}a\-milton--{J}acobi equations and conservation laws}.
\bjtitle{SIAM journal on numerical analysis}
\bvolume{33}(\bissue{4}),
\bfpage{1534}--\blpage{1558}
(\byear{1996})
\end{barticle}
\endbibitem

\bibitem[\protect\citeauthoryear{Lucet}{1996}]{LUCET-96a}
\begin{barticle}
\bauthor{\bsnm{Lucet}, \binits{Y.}}:
\batitle{A fast computational algorithm for the {L}egendre--{F}enchel
  transform}.
\bjtitle{Computational Optimization and Applications}
\bvolume{6}(\bissue{1}),
\bfpage{27}--\blpage{57}
(\byear{1996})
\end{barticle}
\endbibitem

\bibitem[\protect\citeauthoryear{Lucet}{1997}]{LUCET-97b}
\begin{barticle}
\bauthor{\bsnm{Lucet}, \binits{Y.}}:
\batitle{Faster than the {F}ast {L}egendre {T}ransform, the {L}inear-time
  {L}egendre {T}ransform}.
\bjtitle{Numerical Algorithms}
\bvolume{16}(\bissue{2}),
\bfpage{171}--\blpage{185}
(\byear{1997})
\end{barticle}
\endbibitem

\bibitem[\protect\citeauthoryear{Haque and Lucet}{2018}]{HAQUE-18}
\begin{barticle}
\bauthor{\bsnm{Haque}, \binits{T.}},
\bauthor{\bsnm{Lucet}, \binits{Y.}}:
\batitle{A linear-time algorithm to compute the conjugate of convex piecewise
  linear-quadratic bivariate functions}.
\bjtitle{Computational Optimization and Applications}
\bvolume{70}(\bissue{2}),
\bfpage{593}--\blpage{613}
(\byear{2018})
\doiurl{10.1007/s10589-018-0007-1}
\end{barticle}
\endbibitem

\bibitem[\protect\citeauthoryear{Lucet et~al.}{2009}]{LUCET-06}
\begin{barticle}
\bauthor{\bsnm{Lucet}, \binits{Y.}},
\bauthor{\bsnm{Bauschke}, \binits{H.H.}},
\bauthor{\bsnm{Trienis}, \binits{M.}}:
\batitle{The piecewise linear-quadratic model for computational convex
  analysis}.
\bjtitle{Computational Optimization and Applications}
\bvolume{43},
\bfpage{95}--\blpage{118}
(\byear{2009})
\end{barticle}
\endbibitem

\bibitem[\protect\citeauthoryear{Gardiner and Lucet}{2011}]{GARDINER-11a}
\begin{bchapter}
\bauthor{\bsnm{Gardiner}, \binits{B.}},
\bauthor{\bsnm{Lucet}, \binits{Y.}}:
\bctitle{Graph-matrix calculus for computational convex analysis}.
In: \beditor{\bsnm{Bauschke}, \binits{H.H.}},
\beditor{\bsnm{Burachik}, \binits{R.S.}},
\beditor{\bsnm{Combettes}, \binits{P.L.}},
\beditor{\bsnm{Elser}, \binits{V.}},
\beditor{\bsnm{Luke}, \binits{D.R.}},
\beditor{\bsnm{Wolkowicz}, \binits{H.}} (eds.)
\bbtitle{Fixed-Point Algorithms for Inverse Problems in Science and
  Engineering}.
\bsertitle{Springer Optimization and Its Applications},
vol. \bseriesno{49},
pp. \bfpage{243}--\blpage{259}.
\bpublisher{Springer},
\blocation{New York, NY}
(\byear{2011}).
\burl{http://dx.doi.org/10.1007/978-1-4419-9569-8\_12}
\end{bchapter}
\endbibitem

\bibitem[\protect\citeauthoryear{Lucet}{2017}]{LUCET-17}
\begin{botherref}
\oauthor{\bsnm{Lucet}, \binits{Y.}}:
Computational Convex Analysis (CCA) numerical library.
GitHub
(2017)
\end{botherref}
\endbibitem

\bibitem[\protect\citeauthoryear{Lucet}{2021}]{LUCET-21}
\begin{botherref}
\oauthor{\bsnm{Lucet}, \binits{Y.}}:
Computational Convex Analysis for bivariate piecewise linear-cubic functions
  (CCA2) numerical library.
GitHub
(2021)
\end{botherref}
\endbibitem

\bibitem[\protect\citeauthoryear{Hiriart-Urruty and
  Lucet}{2007}]{HIRIART-URRUTY-06}
\begin{barticle}
\bauthor{\bsnm{Hiriart-Urruty}, \binits{J.-B.}},
\bauthor{\bsnm{Lucet}, \binits{Y.}}:
\batitle{Parametric computation of the {L}egendre--{F}enchel conjugate}.
\bjtitle{Journal of Convex Analysis}
\bvolume{14}(\bissue{3}),
\bfpage{657}--\blpage{666}
(\byear{2007})
\end{barticle}
\endbibitem

\bibitem[\protect\citeauthoryear{Lucet}{2006}]{LUCET-05c}
\begin{barticle}
\bauthor{\bsnm{Lucet}, \binits{Y.}}:
\batitle{{F}ast {M}oreau envelope computation {I}: Numerical algorithms}.
\bjtitle{Numerical Algorithms}
\bvolume{43}(\bissue{3}),
\bfpage{235}--\blpage{249}
(\byear{2006})
\doiurl{10.1007/s11075-006-9056-0}
\end{barticle}
\endbibitem

\bibitem[\protect\citeauthoryear{Bauschke et~al.}{2008a}]{BAUSCHKE-07a}
\begin{barticle}
\bauthor{\bsnm{Bauschke}, \binits{H.H.}},
\bauthor{\bsnm{Goebel}, \binits{R.}},
\bauthor{\bsnm{Lucet}, \binits{Y.}},
\bauthor{\bsnm{Wang}, \binits{X.}}:
\batitle{The proximal average: Basic theory}.
\bjtitle{SIAM Journal on Optimization}
\bvolume{19}(\bissue{2}),
\bfpage{768}--\blpage{785}
(\byear{2008})
\end{barticle}
\endbibitem

\bibitem[\protect\citeauthoryear{Bauschke et~al.}{2008b}]{BAUSCHKE-06a}
\begin{barticle}
\bauthor{\bsnm{Bauschke}, \binits{H.H.}},
\bauthor{\bsnm{Lucet}, \binits{Y.}},
\bauthor{\bsnm{Trienis}, \binits{M.}}:
\batitle{How to transform one convex function continuously into another}.
\bjtitle{SIAM Review}
\bvolume{50}(\bissue{1}),
\bfpage{115}--\blpage{132}
(\byear{2008})
\end{barticle}
\endbibitem

\bibitem[\protect\citeauthoryear{Singh and Lucet}{2021}]{SINGH-21}
\begin{barticle}
\bauthor{\bsnm{Singh}, \binits{S.}},
\bauthor{\bsnm{Lucet}, \binits{Y.}}:
\batitle{Linear-time convexity test for low-order piecewise polynomials}.
\bjtitle{{SIAM} Journal on Optimization}
\bvolume{31}(\bissue{1}),
\bfpage{972}--\blpage{990}
(\byear{2021})
\doiurl{10.1137/19M1290851}
{\href{https://arxiv.org/abs/https://doi.org/10.1137/19M1290851}{{https://doi.org/10.1137/19M1290851}}}
\end{barticle}
\endbibitem

\bibitem[\protect\citeauthoryear{Lucet}{2021}]{CCA2}
\begin{botherref}
\oauthor{\bsnm{Lucet}, \binits{Y.}}:
{C}omputational {C}onvex {A}nalysis library 2
(2021).
\url{https://github.com/ylucet/CCA2}
\end{botherref}
\endbibitem

\bibitem[\protect\citeauthoryear{Lucet}{2010}]{LUCET-10}
\begin{barticle}
\bauthor{\bsnm{Lucet}, \binits{Y.}}:
\batitle{What shape is your conjugate? {A} survey of computational convex
  analysis and its applications}.
\bjtitle{SIAM Review}
\bvolume{52}(\bissue{3}),
\bfpage{505}--\blpage{542}
(\byear{2010})
\end{barticle}
\endbibitem

\bibitem[\protect\citeauthoryear{de~Berg et~al.}{2008}]{BERG-08}
\begin{bbook}
\bauthor{\bsnm{Berg}, \binits{M.}},
\bauthor{\bsnm{Kreveld}, \binits{M.}},
\bauthor{\bsnm{Overmars}, \binits{M.}},
\bauthor{\bsnm{Schwarzkopf}, \binits{O.}}:
\bbtitle{Computational Geometry},
\bedition{3}rd edn.,
p. \bfpage{386}.
\bpublisher{Springer},
\blocation{Berlin}
(\byear{2008}).
\bcomment{Algorithms and applications}.
\burl{http://www.springer.com/computer/theoretical+computer+science/book/978-3-540-77973-5}
\end{bbook}
\endbibitem

\bibitem[\protect\citeauthoryear{Kumar and Lucet}{2020}]{KUMAR-19}
\begin{bchapter}
\bauthor{\bsnm{Kumar}, \binits{D.}},
\bauthor{\bsnm{Lucet}, \binits{Y.}}:
\bctitle{Towards the biconjugate of bivariate piecewise quadratic functions}.
In: \beditor{\bsnm{Le~Thi}, \binits{H.A.}},
\beditor{\bsnm{Le}, \binits{H.M.}},
\beditor{\bsnm{Pham~Dinh}, \binits{T.}} (eds.)
\bbtitle{Optimization of Complex Systems: Theory, Models, Algorithms and
  Applications},
pp. \bfpage{257}--\blpage{266}.
\bpublisher{Springer},
\blocation{Cham}
(\byear{2020})
\end{bchapter}
\endbibitem

\bibitem[\protect\citeauthoryear{Karmarkar}{2024}]{KARMARKAR-24}
\begin{botherref}
\oauthor{\bsnm{Karmarkar}, \binits{T.}}:
Computing the conjugate of nonconvex bivariate piecewise linear-quadratic
  functions.
Master's thesis,
University of British Columbia
(2024).
\url{https://dx.doi.org/10.14288/1.0444097}
\end{botherref}
\endbibitem

\bibitem[\protect\citeauthoryear{Rockafellar}{1970}]{ROCKAFELLAR-70}
\begin{barticle}
\bauthor{\bsnm{Rockafellar}, \binits{R.T.}}:
\batitle{On the maximal monotonicity of subdifferential mappings}.
\bjtitle{Pacific Journal of Mathematics}
\bvolume{33},
\bfpage{209}--\blpage{216}
(\byear{1970})
\end{barticle}
\endbibitem

\bibitem[\protect\citeauthoryear{Patrinos and Sarimveis}{2011}]{PATRINOS-11}
\begin{barticle}
\bauthor{\bsnm{Patrinos}, \binits{P.}},
\bauthor{\bsnm{Sarimveis}, \binits{H.}}:
\batitle{Convex parametric piecewise quadratic optimization: Theory and
  algorithms}.
\bjtitle{Automatica}
\bvolume{47}(\bissue{8}),
\bfpage{1770}--\blpage{1777}
(\byear{2011})
\end{barticle}
\endbibitem

\bibitem[\protect\citeauthoryear{Gardiner and Lucet}{2013}]{GARDINER-11}
\begin{barticle}
\bauthor{\bsnm{Gardiner}, \binits{B.}},
\bauthor{\bsnm{Lucet}, \binits{Y.}}:
\batitle{Computing the conjugate of convex piecewise linear-quadratic bivariate
  functions}.
\bjtitle{Mathematical Programming}
\bvolume{139}(\bissue{1-2}),
\bfpage{161}--\blpage{184}
(\byear{2013})
\doiurl{10.1007/s10107-013-0666-8}
\end{barticle}
\endbibitem

\bibitem[\protect\citeauthoryear{Bauschke and Moursi}{2024}]{BAUSCHKE-24}
\begin{bbook}
\bauthor{\bsnm{Bauschke}, \binits{H.H.}},
\bauthor{\bsnm{Moursi}, \binits{W.M.}}:
\bbtitle{An Introduction to Convexity, Optimization, and Algorithms}
vol. \bseriesno{34}.
\bpublisher{Society for Industrial and Applied Mathematics},
\blocation{Philadelphia}
(\byear{2024})
\end{bbook}
\endbibitem

\bibitem[\protect\citeauthoryear{Hiriart-Urruty and
  Lemar{\'e}chal}{1993a}]{HIRIART-URRUTY-93a}
\begin{bbook}
\bauthor{\bsnm{Hiriart-Urruty}, \binits{J.-B.}},
\bauthor{\bsnm{Lemar{\'e}chal}, \binits{C.}}:
\bbtitle{Convex Analysis and Minimization Algorithms I}.
\bsertitle{Grundlehren der Mathematischen Wissenschaften [Fundamental
  Principles of Mathematical Sciences]},
vol. \bseriesno{305}.
\bpublisher{Springer},
\blocation{Berlin}
(\byear{1993}).
\bcomment{Vol I: Fundamentals}.
\burl{http://www.springer.com/math/book/978-3-540-56850-6}
\end{bbook}
\endbibitem

\bibitem[\protect\citeauthoryear{Hiriart-Urruty and
  Lemar{\'e}chal}{1993b}]{HIRIART-URRUTY-93b}
\begin{bbook}
\bauthor{\bsnm{Hiriart-Urruty}, \binits{J.-B.}},
\bauthor{\bsnm{Lemar{\'e}chal}, \binits{C.}}:
\bbtitle{Convex Analysis and Minimization Algorithms {II}}.
\bsertitle{Grundlehren der Mathematischen Wissenschaften [Fundamental
  Principles of Mathematical Sciences]},
vol. \bseriesno{306}.
\bpublisher{Springer},
\blocation{Berlin}
(\byear{1993}).
\bcomment{Vol II: Advanced theory and bundle methods}.
\burl{http://www.springer.com/math/book/978-3-540-56850-6}
\end{bbook}
\endbibitem

\bibitem[\protect\citeauthoryear{Rockafellar}{1970}]{ROCKAFELLAR-70a}
\begin{bbook}
\bauthor{\bsnm{Rockafellar}, \binits{R.T.}}:
\bbtitle{Convex Analysis}.
\bsertitle{Princeton Mathematical Series}.
\bpublisher{Princeton University Press},
\blocation{Princeton, N. J.}
(\byear{1970}).
\burl{http://press.princeton.edu/titles/1815.html}
\end{bbook}
\endbibitem

\bibitem[\protect\citeauthoryear{Kumar}{2019}]{KUMAR-19a}
\begin{botherref}
\oauthor{\bsnm{Kumar}, \binits{D.}}:
Conjugate of some rational functions and convex envelope of quadratic functions
  over a polytope.
Master's thesis,
{U}niversity of {B}ritish {C}olumbia
(2019)
\end{botherref}
\endbibitem

\end{thebibliography}

\end{document}